\documentclass[11pt,a4paper,leqno]{article}
\usepackage{a4wide}
\usepackage{latexsym}
\usepackage{amsmath}
\usepackage{amssymb}
\newtheorem{defin}{Definition}
\newtheorem{lemma}{Lemma}
\newtheorem{prop}{Proposition}
\newtheorem{theo}{Theorem}

\newenvironment{proof}{\medskip\par\noindent{\bf Proof}}{\hfill $\Box$
\medskip\par}

\usepackage{color}

\def\C{\mathbb{C}}
\def\N{\mathbb{N}}
\def\R{\mathbb{R}}

\begin{document}
\title{Strongly regular multi-level solutions of singularly perturbed linear partial differential equations}
\author{{\bf A. Lastra, S. Malek, J. Sanz\footnote{The three authors are partially supported by the research project MTM2012-31439 of Ministerio de Ciencia e Innovaci\'on, Spain}}\\
University of Alcal\'{a}, Departamento de F\'{i}sica y Matem\'{a}ticas,\\
Ap. de Correos 20, E-28871 Alcal\'{a} de Henares (Madrid), Spain\\
University of Lille 1, Laboratoire Paul Painlev\'e,\\
59655 Villeneuve d'Ascq cedex, France\\
University of Valladolid, Dpto. de \'Algebra, An\'alisis Matem\'atico, Geometr{\'\i}a y Topolog{\'\i}a,\\
IMUVA, Paseo de Bel\'en 7, Campus Miguel Delibes, 47011 Valladolid, Spain.\\
{\tt alberto.lastra@uah.es}\\
{\tt Stephane.Malek@math.univ-lille1.fr }\\
{\tt jsanzg@am.uva.es}}
\maketitle

\thispagestyle{empty}

\begin{center}
{\bf Abstract}
\end{center}
We study the asymptotic behavior of the solutions related to a family of singularly perturbed partial differential equations in the complex domain. The analytic solutions are asymptotically represented by a formal power series in the perturbation parameter. The geometry of the problem and the nature of the elements involved in it give rise to different asymptotic levels related to the so-called strongly regular sequences. The result leans on a novel version of a multi-level Ramis-Sibuya theorem.


\medskip


\bigskip

\section{Introduction}

In this work, we study a family of linear partial differential equations of the form
\begin{equation}\label{e1}(\epsilon^{r_1}(t^{k+1}\partial_t)^{s_1}+a)\partial_{z}^{S}X(t,z,\epsilon)=\sum_{(s,\kappa_0,\kappa_1)\in\mathcal{S}}b_{s\kappa_0\kappa_1}(z,\epsilon)t^s(\partial_t^{\kappa_0}\partial_z^{\kappa_1}X)(t,z,\epsilon),
\end{equation}
for given initial data $\partial_{z}^{j}X(t,0,\epsilon)=\phi_{j}(t,\epsilon)$, $0\le j\le S-1$.

Here, we assume $a\in\C^{\star}:=\C\setminus\{0\}$ and $\mathcal{S}$ consists of a finite family of $\N_0^3$, where $\N_0$ stands for the set of nonnegative integers $\{0,1,2,\ldots\}$, $S>\kappa_1$ for every $(s,\kappa_0,\kappa_1)\in\mathcal{S}$ and $b_{s\kappa_0\kappa_1}\in\mathcal{O}(D\times\mathcal{E})$, where $D$ and $\mathcal{E}$  are a neighborhood of the origin and the domain of definition of the perturbation parameter, respectively.

The initial data are provided as holomorphic functions defined in a product of two finite sectors in both, the variable $t$ and the perturbation parameter $\epsilon$.

The shape of the problem under study is analogous to that in~\cite{lama}. The main novelty in the present article consists of the nature of the coefficients $b_{s\kappa_0\kappa_1}$ appearing in the equation. In~\cite{lama}, those belong to $\mathcal{O}\{z,\epsilon\}$, i.e. they are holomorphic functions in a product of neighborhoods of the origin in both $z$ and $\epsilon$ variables. Here, if one writes
$$b_{s\kappa_0\kappa_1}(z,\epsilon)=\sum_{\beta\ge0}b_{s\kappa_0\kappa_1\beta}(\epsilon)\frac{z^{\beta}}{\beta!},$$
for every $(s,\kappa_0,\kappa_1)\in\mathcal{S}$, the function $b_{s\kappa_0\kappa_1\beta}(\epsilon)$ turns out to be the sum of a formal power series $\hat{b}_{s\kappa_0\kappa_1\beta}(\epsilon)$.

This study falls into the framework of the asymptotic analysis of singularly perturbed Cauchy problems of the form $L(t,z,\partial_t,\partial_z,\epsilon)[u(t,z,\epsilon)]=0,$ where $L$ is a linear differential operator, for initial conditions $(\partial_z^ju)(t,0,\epsilon)=h_j(t,\epsilon)$, $0\le j\le \nu-1$ belonging to some functional space.

As a consequence of the unfolding of this problem, two different asymptotic phenomena appear while studying the asymptotic behavior of the analytic solution with respect to its formal one. On the one hand, the appearance of an irregular singularity $t^{k+1}\partial_{t}$ perturbed by a power of $\epsilon$ causes a Gevrey-like asymptotics with respect to the perturbation parameter. Several forbidden directions with respect to summability of the fomal solution would appear in this sense. On the other hand, the nature of the coefficients $b_{s\kappa_0\kappa_1\beta}$ is crucial in the appearance of another different asymptotic behavior of the solution. More precisely, we assume this coefficients are uniquely asymptotically represented by formal power series $\hat{b}_{s\kappa_0\kappa_1\beta}$, with remainder estimates given. The asymptotic representation is given in terms of a more general sequence than Gevrey ones, the so called strongly regular sequences (see Section~\ref{subs51} for the details).

The behavior of the solution of certain problems with respect to the  nature of the elements involved in their equations has been widely studied in the literature. More precisely, one can find several works concerning Gevrey regularity of solutions of ODEs in the Gevrey case (see~\cite{cokr},~\cite{djmi2}) and also involving more general Carleman classes of functions (see~\cite{djmi},~\cite{th}). Likewise, some authors have focused some of their studies on the appearance of Gevrey classes of functions involved in the study of solutions of PDEs (see~\cite{ho},~\cite{si},~\cite{chro},~\cite{ta1},~\cite{ta2}). This work aims to take a step forward in order to deal with Cauchy problems in PDEs in which some of the elements belong to certain general ultraholomorphic classes of functions. The phenomenon observed here combines both Gevrey and a more general behavior related to a strongly regular sequence owing to the inherent structure of the PDE and the nature of the coefficients involved, respectively.

The point of depart of our problem is the formal equation
\begin{equation}\label{e2}(\epsilon^{r_1}(t^{k+1}\partial_t)^{s_1}+a)\partial_{z}^{S}\hat{X}(t,z,\epsilon)=\sum_{(s,\kappa_0,\kappa_1)\in\mathcal{S}}\hat{b}_{s\kappa_0\kappa_1}(z,\epsilon)t^s(\partial_t^{\kappa_0}\partial_z^{\kappa_1}\hat{X})(t,z,\epsilon),
\end{equation}
where
$$\hat{b}_{s\kappa_0\kappa_1}(z,\epsilon)=\sum_{\beta\ge0}\hat{b}_{s\kappa_0\kappa_1\beta}\frac{z^{\beta}}{\beta!},$$
for all $(s,\kappa_0,\kappa_1)\in\mathcal{S}$. After finding an analytic solution of (\ref{e1}), we also provide a formal solution to (\ref{e2}) which are linked by a multisummability procedure, regarding both, a Gevrey order, and the summability with respect to a strongly regular sequence $\mathbb{M}$.

There are other works in the literature which deal with solutions of partial differential equations under the influence of a perturbation parameter, exhibiting singularities of different nature. We refer to the works by M. Canalis-Durand, J. Mozo-Fern\'andez and R. Sch{\"a}fke~\cite{cms}, S. Kamimoto~\cite{kami}, the second author\cite{ma1,ma2} and the authors~\cite{lamasa1}.

Regarding strongly regular sequences as those which govern the asymptotic behavior of the solutions of certain equations, one can find some recent works on this direction. This is the case of~\cite{michalik1},~\cite{michalik2} and ~\cite{balseryoshino}. It is worth remarking that the sequence involved in the $1^{+}$ level (see~\cite{im,im2}), related to the solutions of difference equations, turns out to be a strongly regular sequence which falls out from classical Gevrey ones.

From the authors' point of view, there is no additional difficulties in considering an equation concerning two irregular singularities $(t^{k+1}\partial_t$ perturbed by two powers of $\epsilon$), as it was the case in~\cite{lama}. For the sake of simplicity of the calculations and clarity of the results, we have omitted this more complicated approach.

The strategy followed is to reduce the problem to the study of an auxiliary Cauchy problem by means of Borel transform. A fixed point argument in appropriate weighted Banach spaces allow us to provide a formal solution $W(\tau,z,\epsilon)$ to this auxiliary problem as a formal power series in $z$, say
$$W(\tau,z,\epsilon)=\sum_{\beta\ge0}W_{\beta}(\tau,\epsilon)\frac{z^{\beta}}{\beta!},$$
the coefficients of which belong to certain complex Banach space, what entails appropriate bounds in order to take Laplace transform along well chosen directions in $\tau$ variable. The solution of (\ref{e1}) is then given by
$$X(t,z,\epsilon)=\sum_{\beta\ge0}k\int_{L_{\gamma}}W_{\beta}(u,\epsilon)e^{-\left(\frac{u}{t\epsilon^r}\right)^k}\frac{du}{u}\frac{z^\beta}{\beta!},$$
for $t$ in a bounded open set with vertex at 0, $\mathcal{T}$, and $\epsilon$ in the domain of definition for the perturbation parameter, $\mathcal{E}$. Indeed, the previous formal power series converges near the origin in the variable $z$.

Regarding the asymptotic representation of the analytic solution of (\ref{e1}), we study the difference of two such solutions when varying $\mathcal{E}$ among the elements of a good covering in $\C^{\star}$ (see Definition~\ref{defin1}). We write $(X_{i})_{0\le i\le \nu-1}$ for the finite set of solutions obtained when varying $\mathcal{E}$ among the elements of the good covering. Depending on the geometry of the problem, one can encounter two different cases:
\begin{enumerate}
\item[1-] If there are no singularities with arguments in between the integration lines defining Laplace transform of the two solutions, then the difference of two consecutive solutions in the intersection of the domains of the perturbation parameter is asymptotically flat with respect to the asymptotic behavior coming from the coefficients in the equation (see Theorem~\ref{teo445}).
\item[2-] If there a singularity lies in between the two integration paths defining the Laplace transform of the solutions, then the difference of such solutions is flat of some Gevrey order.
\end{enumerate}

The existence of a formal solution of (\ref{e2}) is obtained from a novel version of Ramis-Sibuya theorem involving both, Gevrey asymptotics and $\mathbb{M}$-asymptotics. We conclude the work with the existence of a formal power series $\hat{X}(t,z,\epsilon)$, which is written as a formal power series in $z$ with coefficients in certain Banach space, $\mathbb{E}$. This formal power series is decomposed into the sum of an element in $\mathbb{E}\{\epsilon\}$, say $a(t,z,\epsilon)$, and $\hat{X}_{1},\hat{X}_2\in\mathbb{E}[[\epsilon]]$. For every $0\le i\le \nu-1$, the actual solution is written in the form
$$ X_{i}(t,z,\epsilon)=a(t,z,\epsilon)+X^1_i(t,z,\epsilon)+X^2_i(t,z,\epsilon).$$
The function $\epsilon\mapsto X^1_i(t,z,\epsilon)$ is an $\mathbb{E}-$valued function which admits $\hat{X}^1(t,z,\epsilon)$ as its $s_1/r_1-$Gevrey asymptotic expansion on $\mathcal{E}_{i}$, and $\epsilon\mapsto X^2_i(t,z,\epsilon)$ is an $\mathbb{E}-$valued function which admits $\hat{X}^2(t,z,\epsilon)$ as its $\mathbb{M}-$asymptotic expansion on $\mathcal{E}_i$, $0\le i\le \nu-1$ (see Theorem~\ref{teo726}). This result can be seen as a multisummability result by generalising the characterisation of multisummability given in~\cite{ba}, Theorem 1, page 57, to the framework of strongly regular sequences.

A plan of the work is the following.

In Section~\ref{seccion1} we define a Banach space of functions and describe some of its properties with respect to some operators acting on it. An auxiliary Cauchy problem is studied in Section~\ref{seccion2}. Its formal solution is obtained by means of a fixed point argument in the space of functions described in the previous section. Section~\ref{seccionotra} is devoted to outline the main properties of Laplace transform and asymptotic expansions and also to describe the analytic solutions of the main problem. In the first part of Section~\ref{seccionotra2} we recall the main definitions and properties of strongly regular sequences, asymptotic expansions and summability whilst in the second part, we provide the information for the flatness of the difference of two consecutive solutions in the perturbation parameter. The work concludes in Section~\ref{seccionotramas} with the development of a Ramis-Sibuya-like theorem in two levels, the existence of a formal solution to the main problem, and the asymptotic relationship between the analytic solution and the formal one.

\section{Banach spaces of functions with exponential decay}\label{seccion1}

Let $\rho_0>0$. We write $D(0,\rho_0)\subseteq\C$ for the open disc with center at 0 and radius $\rho_0$. For $d\in\R$, we consider an unbounded sector $S_d:=\{z\in\C:|\arg(z)-d|<\delta_1\}$ for some $\delta_1>0$. We put $\Omega:= S_d\cup D(0,\rho_0)$.

Let $\mathcal{E}$ be an open and bounded sector with vertex at 0.

Throughout this work, $b$ and $\sigma$ are fixed positive real numbers with $b>1$, whilst $k\ge2$ stands for a fixed integer.

\begin{defin}\label{defi74}
Let $\epsilon\in\mathcal{E}$ and $r\in\mathbb{Q}$, $r>0$.

For every $\beta\ge0$, we consider the vector space $F_{\beta,\epsilon,\Omega}$, consisting of the holomorphic functions defined in $\Omega$, $\tau\mapsto h(\tau,\epsilon)$ such that
$$\left\|h(\tau,\epsilon)\right\|_{\beta,\epsilon,\Omega}:=\sup_{\tau\in\Omega}\left\{\frac{1+\left|\frac{\tau}{\epsilon^r}\right|^{2k}}{\left|\frac{\tau}{\epsilon^r}\right|}\exp\left(-\sigma r_b(\beta)\left|\frac{\tau}{\epsilon^r}\right|^{k}\right)|h(\tau,\epsilon)| \right\}<\infty,$$
where $r_b(\beta)=\sum_{n=0}^{\beta}(n+1)^{-b}$. It is straightforward to check that $(F_{\beta,\epsilon,\Omega},\left\|\cdot\right\|_{\beta,\epsilon,\Omega})$ is a Banach space.
\end{defin}

The previous definition is motivated by the corresponding one in~\cite{lama}. There, the domain $\Omega$ did depend on $\epsilon$ whilst in the present work it does not. This dependence was caused by the existence of a movable singularity, not appearing in the current context.

The forthcoming assumption and results are analogous to those in~\cite{lama}, so we omit the details. Assumption (A) guarantees the existence of a positive distance from the elements in $\Omega$ and the singularity coming from the equation. The following results are concerned with the behavior of some operators when acting on the elements in $F_{\beta,\epsilon,\Omega}$.

\textbf{Assumption (A):} Let $a\in\C$ with $a\neq0$, and let $s_1$ be a positive integer. We assume that $ks_1\arg(\tau)\neq \pi(2j+1)+\arg(a)$ for $j=0,\ldots,ks_1-1$ and every $\tau\in\overline{S_d}\setminus\{0\}$. In addition to this, we take $\rho_0$ under the condition $\rho_0<\frac{|a|^{1/ks_1}}{2k^{1/k}}$.

\begin{lemma}\label{lema1}(Lemma 1,~\cite{lama})
Under Assumption (A), there exists a constant $C_1>0$ (which only depends on $k$, $s_1$, $a$) such that
$$\left|\frac{1}{(k\tau^k)^{s_1}+a}\right|\le C_1,$$
for every $\tau\in\Omega$.
\end{lemma}

\begin{lemma}\label{lema2}(Lemma 3,~\cite{lama})
Let $\epsilon\in\mathcal{E}$ and $\beta$ be a nonnegative integer. Given any bounded and holomorphic function $g(\tau)$ on $\Omega$, then
$$\left\|g(\tau)h(\tau,\epsilon)\right\|_{\beta,\epsilon,\Omega}\le M_{g}\left\|h(\tau,\epsilon)\right\|_{\beta,\epsilon,\Omega},$$
for every $h\in F_{\beta,\epsilon,\Omega}$. Here, the constant $M_{g}$ is defined by $\sup_{\tau\in\Omega}|g(\tau)|$.
\end{lemma}

\begin{prop}\label{prop1}(Proposition 1,~\cite{lama})
Let $\epsilon\in\mathcal{E}$ and $r\in\mathbb{Q}$, $r>0$. We consider real numbers $\nu\ge0$ and $\xi\ge -1$. We also fix nonnegative integers $S,\alpha,\beta$ with $S\ge1$ and $\alpha<\beta$. Then, there exists a constant $C_2>0$ (depending on $\alpha$, $S$, $\beta$, $\xi$, $\nu$ and which does not depend on $\epsilon$) with
$$\left\|\tau^k\int_0^{\tau^k}(\tau^k-s)^{\nu}s^{\xi}f(s^{1/k},\epsilon)ds\right\|_{\beta,\epsilon,\Omega}\le C_2|\epsilon|^{rk(2+\nu+\xi)}\left(\frac{(\beta+1)^b}{\beta-\alpha}\right)^{\nu+\xi+3}\left\|f(\tau,\epsilon)\right\|_{\alpha,\epsilon,\Omega},$$
for every $f\in F_{\alpha,\epsilon,\Omega}$.
\end{prop}

\section{An auxiliary Cauchy problem}\label{seccion2}

Let $S_{d}$ and $\Omega$ be constructed as in Section~\ref{seccion1}. We also preserve the definition of $\mathcal{E}$ and the choice of the constants $a$, $s_1$.

Let $S$, $r_1$ be positive integers, and put
\begin{equation}\label{e111}
r:=\frac{r_1}{s_1k}.
\end{equation}
$\mathcal{S}$ stands for a finite subset of $\N_0^3$ under the next

\textbf{Assumption (B):} For every $(s,\kappa_0,\kappa_1)\in\mathcal{S}$ we have $\kappa_1<S$. Moreover, there exists a nonnegative integer $\delta_{\kappa_0}\ge k$ such that
$$s=\kappa_0(k+1)+\delta_{\kappa_0}.$$

For every $(s,\kappa_0,\kappa_1)\in\mathcal{S}$, $b_{s\kappa_0\kappa_1}(z,\epsilon)$ is a holomorphic and bounded function defined in $D\times \mathcal{E}$, for some neighborhood of the origin, $D$. We put
\begin{equation*}
b_{s\kappa_0\kappa_1}(z,\epsilon)=\sum_{\beta\ge0}b_{s\kappa_0\kappa_1\beta}(\epsilon)\frac{z^{\beta}}{\beta!},
\end{equation*}
with $b_{s\kappa_0\kappa_1\beta}\in\mathcal{O}(\mathcal{E})$ for every $\kappa_0,\kappa_1$ such that $(s,\kappa_0,\kappa_1)\in\mathcal{S}$ and all $\beta\ge0$.

We also fix constants $A_{\kappa_0,p}\in\C$ for every $(s,\kappa_0,\kappa_1)\in\mathcal{S}$ and $1\le p\le \kappa_0$; their meaning will be made clear in the proof of Theorem \ref{teo280}.

For every $\epsilon\in\mathcal{E}$ we study the Cauchy problem
\begin{equation}\label{e127} ((k\tau^{k})^{s_1}+a)\partial_z^SW(\tau,z,\epsilon)\phantom{mmmmmmmmmmmmmmmmmmmmmmmmmmm}
\end{equation}
$$=\sum_{(s,\kappa_0,\kappa_1)\in\mathcal{S}}b_{s\kappa_0\kappa_1}(z,\epsilon)\epsilon^{-r(s-\kappa_0)} \left[\frac{\tau^{k}}{\Gamma\left(\frac{\delta_{\kappa_0}}{k}\right)} \int_{0}^{\tau^k}(\tau^k-\sigma)^{\frac{\delta_{\kappa_0}}{k}-1}(k\sigma)^{\kappa_0} \partial_z^{\kappa_1}W(\sigma^{1/k},z,\epsilon)\frac{d\sigma}{\sigma} \right.$$
$$ \left. +\sum_{1\le p\le \kappa_0-1}A_{\kappa_0,p}\frac{\tau^k}{\Gamma\left(\frac{\delta_{\kappa_0}+k(\kappa_0-p)}{k} \right)}\int_{0}^{\tau^k}(\tau^k-\sigma)^{\frac{\delta_{\kappa_0}+k(\kappa_0-p)}{k}-1}(k\sigma)^{p} \partial_{z}^{\kappa_1}W(\sigma^{1/k},z,\epsilon)\frac{d\sigma}{\sigma}\right],$$
for initial data
\begin{equation}\label{e128}
(\partial_z^j)W(\tau,0,\epsilon)=W_j(\tau,\epsilon)\in F_{j,\epsilon,\Omega},\quad 0\le j\le S-1.
\end{equation}

We assume that
\begin{equation}\label{e322}
\sup_{\epsilon\in\mathcal{E}}\left\|W_j(\tau,\epsilon)\right\|_{j,\epsilon,\Omega}<\infty, \quad 0\le j\le S-1.
\end{equation}

\begin{prop}\label{prop150}
Under Assumptions (A), (B) and (\ref{e322}), there exists a formal power series solution of (\ref{e127}), (\ref{e128}),
\begin{equation}\label{e139}
W(\tau,z,\epsilon)=\sum_{\beta\ge0}W_{\beta}(\tau,\epsilon)\frac{z^{\beta}}{\beta!},
\end{equation}
where $W_{\beta}(\tau,\epsilon)\in F_{\beta,\epsilon,\Omega}$ for every $\epsilon\in\mathcal{E}$, $\beta\ge0$.
Moreover, there exist  $Z_0,Z_1>0$ such that
\begin{equation}\label{e194}
\left\|W_{\beta}(\tau,\epsilon)\right\|_{\beta,\epsilon,\Omega}\le Z_1Z_0^{\beta}\beta!
\end{equation}
for every $\beta\ge0$ and every $\epsilon\in\mathcal{E}$.
\end{prop}
\begin{proof}
Let $\epsilon\in\mathcal{E}$, $\tau\in\Omega$ and $\beta\ge0$. Taking the formal power series (\ref{e139}) into (\ref{e127}) one gets that a formal solution of the problem satisfies the recursion formula
\begin{equation}\label{e145}
\frac{W_{\beta+S}(\tau,\epsilon)}{\beta!}=\frac{1}{(k\tau^k)^{s_1}+a} \sum_{(s,\kappa_0,\kappa_1)\in\mathcal{S}}\sum_{\alpha_0+\alpha_1=\beta} \frac{b_{s\kappa_0\kappa_1\alpha_0}(\epsilon)}{\alpha_0!}\epsilon^{-r(s-\kappa_0)}\times
\end{equation}
$$\times\left[\frac{\tau^{k}}{\Gamma\left(\frac{\delta_{\kappa_0}}{k}\right)} \int_{0}^{\tau^k}(\tau^k-\sigma)^{\frac{\delta_{\kappa_0}}{k}-1}(k\sigma)^{\kappa_0} \frac{W_{\alpha_1+\kappa_1}(\sigma^{1/k},\epsilon)}{\alpha_1!}\frac{d\sigma}{\sigma}\right.$$
$$\left.+\sum_{1\le p\le \kappa_0-1}A_{\kappa_0,p} \frac{\tau^{k}}{\Gamma\left(\frac{\delta_{\kappa_0}+k(\kappa_0-p)}{k}\right)} \int_{0}^{\tau^{k}}(\tau^k-\sigma)^{\frac{\delta_{\kappa_0}+k(\kappa_0-p)}{k}-1}(k\sigma)^p \frac{W_{\alpha_1+\kappa_1}(\sigma^{1/k},\epsilon)}{\alpha_1!}\frac{d\sigma}{\sigma}  \right].$$
Taking into account Lemma~\ref{lema1}, the function $\tau\mapsto W_{\beta+S}(\tau,\epsilon)$ is well defined and holomorphic in $\Omega$. We now prove it turns out to be an element in $F_{\beta,\epsilon,\Omega}$.

If $0\le \beta\le S-1$, the statement is true from the choice of the initial conditions in (\ref{e128}). We assume $\beta\ge S$ and put $w_{\beta}(\epsilon):=\left\|W_\beta(\tau,\epsilon)\right\|_{\beta,\epsilon,\Omega}$. By taking norms $\left\|\cdot\right\|_{\beta+S,\epsilon,\Omega}$ at both sides of (\ref{e145}) one obtains the inequality
$$\frac{w_{\beta+S}(\epsilon)}{\beta!}\le \frac{1}{|(k\tau^k)^{s_1}+a|}\sum_{(s,\kappa_0,\kappa_1)\in\mathcal{S}}\sum_{\alpha_0+\alpha_1=\beta}\frac{|b_{s\kappa_0\kappa_1\alpha_0}(\epsilon)|}{\alpha_0!}|\epsilon|^{-r(s-\kappa_0)}\times$$
$$\times\left[\left\|  \frac{\tau^k}{\Gamma\left(\frac{\delta_{\kappa_0}}{k}\right)} \int_{0}^{\tau^{k}}(\tau^k-\sigma)^{\frac{\delta_{\kappa_0}}{k}-1}(k\sigma)^{\kappa_0} \frac{W_{\alpha_1+\kappa_1}(\sigma^{1/k},\epsilon)}{\alpha_1!}\frac{d\sigma}{\sigma}      \right\|_{\beta+S,\epsilon,\Omega}\right.$$
$$\left.+\sum_{1\le p\le \kappa_0-1}|A_{\kappa_0,p}| \left\|\frac{\tau^k}{\Gamma\left(\frac{\delta_{\kappa_0}+k(\kappa_0-p)}{k}\right)} \int_{0}^{\tau^k}(\tau^k-\sigma)^{\frac{\delta_{\kappa_0}+k(\kappa_0-p)}{k}-1}(k\sigma)^{p} \frac{W_{\alpha_1+\kappa_1}(\sigma^{1/k},\epsilon)}{\alpha_1!}\frac{d\sigma}{\sigma} \right\|_{\beta+S,\epsilon,\Omega}\right].$$
Regarding Lemma~\ref{lema1} and Proposition~\ref{prop1}, one can write
\begin{equation*}
\frac{w_{\beta+S}(\epsilon)}{\beta!}\le C_1\sum_{(s,\kappa_0,\kappa_1)\in \mathcal{S}}\sum_{\alpha_0+\alpha_1=\beta} \frac{|b_{s\kappa_0\kappa_1\alpha_0}(\epsilon)|}{\alpha_0!}|\epsilon|^{-r(s-\kappa_0)} |\epsilon|^{rk(\frac{\delta_{\kappa_0}}{k}+\kappa_0)} \left(\frac{(\beta+S+1)^b}{\beta+S-\alpha_1-\kappa_1}\right)^{\frac{\delta_{\kappa_0}}{k}+\kappa_0+1}
\end{equation*}
$$\times\left[\frac{C_2k^{\kappa_0}}{\Gamma\left(\frac{\delta_{\kappa_0}}{k}\right)}+\sum_{1\le p\le \kappa_0-1}|A_{\kappa_0,p}|\frac{C_2k^p}{\Gamma\left(\frac{\delta_{\kappa_0}+k(\kappa_0-p)}{k}\right)}\right]\frac{w_{\alpha_1+\kappa_1}(\epsilon)}{\alpha_1!}.$$
Assumption (B) leads to
$$|\epsilon|^{\delta_{\kappa_0}+\kappa_0(k+1)}\le C_3$$
for some $C_3>$ which does not depend on $\epsilon\in\mathcal{E}$ nor $\kappa_0$. In addition to this, there exists $C_4>0$, not depending on $\epsilon\in\mathcal{E}$ such that

$$\left(\frac{(\beta+S+1)^b}{\beta+S-\alpha_1-\kappa_1}\right)^{\frac{\delta_{\kappa_0}+k(\kappa_0-p)}{k}+p+1}\le C_4\frac{\beta!}{\left(\beta-\left\lfloor b(\delta_{\kappa_0}/k+\kappa_0+1)\right\rfloor +1 \right)!}.$$

Moreover, the domain of holomorphy for $b_{s\kappa_0\kappa_1}$ provides the existence of a positive constant $C_5$, not depending on $\epsilon$ such that

$$|b_{s\kappa_0\kappa_1\beta}(\epsilon)|\le C_5^{\beta}\beta!=:B_{s\kappa_0\kappa_1\beta},$$
for every $\epsilon\in\mathcal{E}$.

Let $B_{s\kappa_0\kappa_1}(z)=\sum_{\beta\ge0}B_{s\kappa_0\kappa_1\beta}\frac{z^{\beta}}{\beta!}$, which defines a holomorphic function in some neighborhood of the origin. We consider the Cauchy problem
\begin{align}
\partial_x^Su(x,\epsilon)&=C_1C_2C_3C_4\sum_{(s,\kappa_0,\kappa_1)\in\mathcal{S}}\left[\frac{k^{\kappa_0}}{\Gamma\left(\frac{\delta_{\kappa_0}}{k}\right)}+\sum_{1\le p\le \kappa_0-1}|A_{\kappa_0,p}|\frac{k^{p}}{\Gamma\left(\frac{\delta_{\kappa_0}+k(\kappa_0-p)}{k}\right)}\right]\label{e174}\\
& \times x^{\left\lfloor b(\delta_{\kappa_0}/k+\kappa_0+1)\right\rfloor-1}\partial_{x}^{\left\lfloor b(\delta_{\kappa_0}/k+\kappa_0+1)\right\rfloor-1}\left(B_{s\kappa_0\kappa_1}(x)\partial_x^{\kappa_1}u(x,\epsilon)\right),\nonumber
\end{align}
with initial conditions
\begin{equation}\label{e175}(\partial_x^ju)(0,\epsilon)=w_{j}(\epsilon),\quad 0\le j\le S-1.\end{equation}

One may check that the problem (\ref{e174}), (\ref{e175}) has a unique formal solution
$$u(x,\epsilon)=\sum_{\beta\ge0}u_{\beta}(\epsilon)\frac{x^\beta}{\beta!}\in\mathbb{R}[[x]],$$
whose coefficients satisfy the recursion formula
\begin{align*}
\frac{u_{\beta+S}(\epsilon)}{\beta!}&=C_1C_2C_3C_4\sum_{(s,\kappa_0,\kappa_1)\in \mathcal{S}}\sum_{\alpha_0+\alpha_1=\beta}\frac{B_{s\kappa_0\kappa_1\alpha_0}}{\alpha_0!} \frac{\beta!}{\left(\beta-\left\lfloor b(\delta_{\kappa_0}/k+\kappa_0+1)\right\rfloor+1\right)!}\\
&\times\left[\frac{k^{\kappa_0}}{\Gamma\left(\frac{\delta_{\kappa_0}}{k}\right)}+\sum_{1\le p\le \kappa_0-1}|A_{\kappa_0,p}|\frac{k^p}{\Gamma\left(\frac{\delta_{\kappa_0}+k(\kappa_0-p)}{k}\right)}\right] \frac{u_{\alpha_1+\kappa_1}(\epsilon)}{\alpha_1!}
\end{align*}
for every $\beta\ge0$. It is clear from the initial conditions that $u_j(\epsilon)=w_j(\epsilon)$ for $0\le j \le S-1$. Regarding the construction of the formal power series $u(x,\epsilon)$ one concludes that $$w_{\beta}(\epsilon)\le u_{\beta}(\epsilon),\qquad \beta\ge0.$$
Due to assumption (\ref{e322}), the classical theory of existence of solutions of ODEs can be applied to the problem (\ref{e174}), (\ref{e175}) in order to guarantee that the formal power series $u(x,\epsilon)$ is convergent in a neighborhood of the origin, and the radius of convergence does not depend on the choice of $\epsilon\in\mathcal{E}$. So, there exist $Z_0,Z_1>0$ such that
$$\sum_{\beta\ge0}u_{\beta}(\epsilon)\frac{Z_0^{-\beta}}{\beta!}<Z_1,$$
for all $\epsilon\in\mathcal{E}$. As a consecuence, $0<u_{\beta}(\epsilon)<Z_1Z_0^\beta\beta!$ for every $\beta\ge0$, and
\begin{equation*}
\left\|W_{\beta}(\tau,\epsilon)\right\|_{\beta,\epsilon,\Omega}=w_{\beta}(\epsilon)\le u_{\beta}(\epsilon)\le Z_1Z_0^{\beta}\beta!<\infty,
\end{equation*}
for every $\beta\ge0$. This concludes the proof.
\end{proof}

\section{Analytic solutions of the main Cauchy problem}\label{seccionotra}

In the present section, we state the main singular Cauchy problem under study in this work, and provide an analytic solution for it. The procedure rests on the $k-$Borel summability method of formal power series which is a slightly modified version of the more classical procedure (see~\cite{ba2}, Section 3.2). This novel version has already been used in works such as~\cite{lama1} and~\cite{lama} when studying Cauchy problems under the presence of a small perturbation parameter.

\subsection{Laplace transform and asymptotic expansions}\label{sec42}

\begin{defin}\label{defi2}
Let $k\ge1$ be an integer. Let $(m_{k}(n))_{n\ge1}$ be the sequence
$$m_{k}(n)=\Gamma\left(\frac{n}{k}\right)=\int_{0}^{\infty}t^{\frac{n}{k}-1}e^{-t}dt,\qquad n\ge1.$$
Let $(\mathbb{E},\left\|\cdot\right\|_{\mathbb{E}})$ be a complex Banach space. We say a formal power series
$$\hat{X}(T)=\sum_{n=1}^{\infty}a_{n}T^{n}\in T\mathbb{E}[[T]]$$ is $m_{k}$-summable with respect to $T$ in the direction $d\in[0,2\pi)$ if the following assertions hold:
\begin{enumerate}
\item There exists $\rho>0$ such that the $m_{k}$-Borel transform of $\hat{X}$, $\mathcal{B}_{m_{k}}(\hat{X})$, is absolutely convergent for $|\tau|<\rho$, where
$$\mathcal{B}_{m_{k}}(\hat{X})(\tau)=\sum_{n=1}^{\infty}\frac{a_{n}}{\Gamma\left(\frac{n}{k}\right)}\tau^{n}\in\tau\mathbb{E}[[\tau]].$$
\item The series $\mathcal{B}_{m_{k}}(\hat{X})$ can be analytically continued in a sector $S=\{\tau\in\C^{\star}:|d-\arg(\tau)|<\delta\}$ for some $\delta>0$. In addition to this, the extension is of exponential growth at most $k$ in $S$, meaning that there exist $C,K>0$ such that
$$\left\|\mathcal{B}_{m_{k}}(\hat{X})(\tau)\right\|_{\mathbb{E}}\le Ce^{K|\tau|^{k}},\quad \tau\in S.$$
\end{enumerate}
Under these assumptions, the vector valued Laplace transform of $\mathcal{B}_{m_{k}}(\hat{X})$ along direction $d$ is defined by
$$\mathcal{L}_{m_{k}}^{d}\left(\mathcal{B}_{m_{k}}(\hat{X})\right)(T)=k\int_{L_{\gamma}}\mathcal{B}_{m_{k}}(\hat{X})(u)e^{-(u/T)^k}\frac{du}{u},$$
where $L_{\gamma}$ is the path parametrized by $u\in[0,\infty)\mapsto ue^{i\gamma}$, for some appropriate direction $\gamma$ depending on $T$, such that $L_{\gamma}\subseteq S$ and $\cos(k(\gamma-\arg(T)))\ge\Delta>0$ for some $\Delta>0$.

The function $\mathcal{L}_{m_{k}}^{d}(\mathcal{B}_{m_{k}}(\hat{X})$ is well defined and turns out to be a holomorphic and bounded function in any sector of the form $S_{d,\theta,R^{1/k}}=\{T\in\C^{\star}:|T|<R^{1/k},|d-\arg(T)|<\theta/2\}$, for some $\frac{\pi}{k}<\theta<\frac{\pi}{k}+2\delta$ and $0<R<\Delta/K$. This function is known as the $m_k$-sum of the formal power series $\hat{X}(T)$ in the direction $d$.
\end{defin}

The following are some elementary properties concerning the $m_k$-sums of formal power series which will be crucial in our procedure.

1) The function $\mathcal{L}_{m_{k}}^{d}(\mathcal{B}_{m_{k}}(\hat{X}))(T)$ admits $\hat{X}(T)$ as its Gevrey asymptotic expansion of order $1/k$ with respect to $t$ in $S_{d,\theta,R^{1/k}}$. More precisely, for every $\frac{\pi}{k}<\theta_1<\theta$, there exist $C,M>0$ such that
$$\left\|\mathcal{L}^{d}_{m_{k}}(\mathcal{B}_{m_{k}}(\hat{X}))(T)-\sum_{p=1}^{n-1}a_{p}T^{p}\right\|_{\mathbb{E}}\le CM^{n}\Gamma(1+\frac{n}{k})|T|^{n},$$
for every $n\ge2$ and $T\in S_{d,\theta,R^{1/k}}$. Watson's lemma (see Proposition 11 p.75 in \cite{ba}) allows us to affirm that $\mathcal{L}^{d}_{m_{k}}(\mathcal{B}_{m_{k}}(\hat{X})(T)$ is unique provided that the opening $\theta_1$ is larger than $\frac{\pi}{k}$.

2) Whenever $\mathbb{E}$ is a Banach algebra, the set of holomorphic functions having Gevrey asymptotic expansion of order $1/k$ on a sector with values in $\mathbb{E}$ turns out to be a differential algebra (see Theorem 18, 19 and 20 in~\cite{ba}). This, and the uniqueness provided by Watson's lemma allow us to obtain some properties on $m_{k}$-summable formal power series in direction $d$.

By $\star$ we denote the product in the Banach algebra and also the Cauchy product of formal power series with coefficients in $\mathbb{E}$. Let $\hat{X}_{1}$, $\hat{X}_{2}\in T\mathbb{E}[[T]]$ be $m_{k}$-summable formal power series in direction $d$. Let $q_1\ge q_2\ge1$ be integers. Then $ \hat{X}_{1}+\hat{X}_{2}$, $\hat{X}_{1}\star \hat{X}_{2}$ and $T^{q_1}\partial_{T}^{q_2}\hat{X}_{1}$, which are elements of $T\mathbb{E}[[T]]$, are $m_{k}$-summable in direction $d$. Moreover, one has
$$\mathcal{L}_{m_{k}}^{d}(\mathcal{B}_{m_{k}}(\hat{X}_{1}))(T)+\mathcal{L}_{m_{k}}^{d}(\mathcal{B}_{m_{k}}(\hat{X}_{2}))(T)=\mathcal{L}_{m_{k}}^{d}(\mathcal{B}_{m_{k}}(\hat{X}_{1}+\hat{X}_{2}))(T),$$
$$\mathcal{L}_{m_{k}}^{d}(\mathcal{B}_{m_{k}}(\hat{X}_{1}))(T)\star \mathcal{L}_{m_{k}}^{d}(\mathcal{B}_{m_{k}}(\hat{X}_{2}))(T)=\mathcal{L}_{m_{k}}^{d}(\mathcal{B}_{m_{k}}(\hat{X}_{1}\star\hat{X}_{2}))(T),$$
$$T^{q_1}\partial_{T}^{q_2}\mathcal{L}^{d}_{m_{k}}(\mathcal{B}_{m_{k}}(\hat{X}_{1}))(T)=\mathcal{L}_{m_{k}}^{d}(\mathcal{B}_{m_{k}}(T^{q_1}\partial_{T}^{q_2}\hat{X}_{1}))(T),$$
for every $T\in S_{d,\theta,R^{1/k}}$.

The next proposition is written without proof for it can be found in \cite{lama1}, Proposition 6.

\begin{prop}\label{prop3}
Let $\hat{f}(t)=\sum_{n\ge1}f_nt^n\in\mathbb{E}[[t]]$, where $(\mathbb{E},\left\|\cdot\right\|_{\mathbb{E}})$ is a Banach algebra. Let $k,m\ge1$ be integers. The following formal identities hold.
$$\mathcal{B}_{m_{k}}(t^{k+1}\partial_{t}\hat{f}(t))(\tau)=k\tau^{k}\mathcal{B}_{m_{k}}(\hat{f}(t))(\tau),$$
$$\mathcal{B}_{m_{k}}(t^{m}\hat{f}(t))(\tau)=\frac{\tau^{k}}{\Gamma\left(\frac{m}{k}\right)}\int_{0}^{\tau^{k}}(\tau^{k}-s)^{\frac{m}{k}-1}\mathcal{B}_{m_{k}}(\hat{f}(t))(s^{1/k})\frac{ds}{s}.$$
\end{prop}

\subsection{Analytic solutions of a singular Cauchy problem}

Let $S,s_1\ge1$ $r_1\ge0$ be integers and $a\in\C^{\star}$. We define $r$ as in (\ref{e111}). Provided that Assumption (A) holds, we consider the construction made in Section~\ref{seccion1} for $\mathcal{E}$, $S_d$ (also $\delta_1$) and $D(0,\rho_0)$. We take $\gamma\in[0,2\pi)$ such that $\R_+e^{i\gamma}\subseteq S_d\cup\{0\}$.

Let $\mathcal{S}$ be as in Section~\ref{seccion2} and satisfying Assumption (B). For every $(s,\kappa_0,\kappa_1)\in \mathcal{S}$ we take the function $b_{s\kappa_0\kappa_1}(z,\epsilon)$ and $A_{\kappa_0 p}\in\C$ for all $0\le p\le \kappa_0-1$ as those in Section~\ref{seccion2}.

Let $\epsilon\in\mathcal{E}$. We consider the Cauchy problem
\begin{equation}\label{e257}
((T^{k+1}\partial_T)^{s_1}+a)\partial_z^{S}Y(T,z,\epsilon)=\sum_{(s,\kappa_0,\kappa_1)\in\mathcal{S}}b_{s\kappa_0\kappa_1}(z,\epsilon)\epsilon^{-r(s-\kappa_0)}T^{s}(\partial_T^{\kappa_0}\partial_z^{\kappa_1}Y)(T,z,\epsilon),
\end{equation}
under the initial conditions
\begin{equation}\label{e258}
(\partial_{z}^{j}Y)(T,0,\epsilon)=Y_{j}(T,\epsilon),\quad 0\le j\le S-1
\end{equation}
constructed as follows:
For every $0\le j\le S-1$, one considers a holomorphic function $\tau\mapsto W_{j}(\tau,\epsilon)$ defined in $\Omega$. Moreover, we assume (\ref{e322}) holds, and construct $Y_{j}(T,\epsilon):=\mathcal{L}_{m_{k}}^d(W_{j}(\tau,\epsilon))(T)$,
as the Laplace transform with respect to the variable $\tau$, along direction $d$. Condition 2 in Definition~\ref{defi2} is fulfilled so that the previous definition makes sense, providing a function $T\mapsto Y_{j}(T,\epsilon)$ holomorphic in the set of all $T=|T|e^{i\theta}$ such that $\cos(k(\gamma-\theta))\ge \Delta$ for some $\Delta>0$, and $|T|\le |\epsilon|^{r}\left(\frac{\Delta}{\sigma \xi(b)}\right)^{1/k}$, with $\xi(b)=\sum_{n\ge0}\frac{1}{(n+1)^b}$.

The next result follows analogous steps as the corresponding one of Theorem 1 in~\cite{lama}. It is based on an expansion of the operators involved which can be found in~\cite{taya}.

\begin{theo}\label{teo280}
Let $\epsilon\in\mathcal{E}$. Departing from the constructions within the present subsection, problem (\ref{e257}), (\ref{e258}) admits a holomorphic solution
\begin{equation}\label{e282}
(T,z)\mapsto Y(T,z,\epsilon):=\sum_{\beta\ge0}\mathcal{L}_{m_{k}}^{d}(W_{\beta}(\tau,\epsilon))(T)\frac{z^{\beta}}{\beta!}
\end{equation}
defined in
$$S_{d,\theta, |\epsilon|^{r}\left(\frac{\Delta}{\sigma\xi(b)}\right)^{1/k}}\times D(0,\frac{1}{Z_0}),$$
where $Z_0$ is introduced in (\ref{e194}), and $$S_{d,\theta,|\epsilon|^r\left(\frac{\Delta}{\sigma \xi(b)}\right)^{1/k}}=\left\{T\in\C^{\star}:|T|\le |\epsilon|^r\left(\frac{\Delta}{\sigma \xi(b)}\right)^{1/k},|\arg (T)-d|<\frac{\theta}{2}\right\}$$
for some $\theta>\pi/k$. The Laplace transform in (\ref{e282}) is taken with respect to $\tau$ variable.
\end{theo}

\begin{proof}
From Assumption (B), one can write $T^s\partial_{T}^{\kappa_0}$ in the form $T^{\delta_{\kappa_0}}T^{\kappa_0(k+1)}\partial_{T}^{\kappa_0}$ for every $(s,\kappa_0,\kappa_1)\in\mathcal{S}$, for some $\delta_{\kappa_0}>0$. The expansion formula in page 40 of~\cite{taya} allows us to write
\begin{equation*}
T^{\delta_{\kappa_0}}T^{\kappa_0(k+1)}\partial_{T}^{\kappa_0}=T^{\delta_{\kappa_0}}\left((T^{k+1}\partial_T)^{\kappa_0}+\sum_{1\le p\le \kappa_0-1}A_{\kappa_0,p}T^{k(\kappa_0-p)}(T^{k+1}\partial_{T})^p\right),
\end{equation*}
for some $A_{\kappa_0,p}\in\C$. Equation (\ref{e257}) is transformed into
\begin{equation}\label{e347}((T^{k+1}\partial_T)^{s_1}+a)\partial_z^SY(T,z,\epsilon)\end{equation}
$$
=\sum_{(s,\kappa_0,\kappa_1)\in\mathcal{S}}b_{\kappa_0\kappa_1}(z,\epsilon)\epsilon^{-r(s-\kappa_0)}T^{\delta_{\kappa_0}}\left[(T^{k+1}\partial_{T})^{\kappa_0}+\sum_{1\le p\le \kappa_0-1}A_{\kappa_0,p}T^{k(\kappa_0-p)}(T^{k+1}\partial_{T})^p\right]\partial_{z}^{\kappa_1}Y(T,z,\epsilon).
$$
By means of the formal Borel transform with respect to $T$, and bearing in mind the properties of such operator (see Proposition~\ref{prop3}), equation (\ref{e347}) is transformed into (\ref{e127}), with $W(\tau,z,\epsilon)=\mathcal{B}_{m_{k}}(Y(T,z,\epsilon))(\tau)$. In view of (\ref{e322}), one has $W_j\in F_{j,\epsilon,\Omega}$ for $0\le j\le S-1$, and moreover we may apply Proposition~\ref{prop150} to the Cauchy problem with equation (\ref{e127}) and initial data \begin{equation}\label{e127a}
(\partial_z^jW)(\tau,0,\epsilon)=W_j(\tau,\epsilon),\quad 0\le j\le S-1.
\end{equation}
Consequently, there exists a formal solution of the previous Cauchy problem of the form
\begin{equation}\label{e352a}
\sum_{\beta\ge0}W_{\beta}(\tau,\epsilon)\frac{z^\beta}{\beta!},
\end{equation}
with $W_{\beta}(\tau,\epsilon)\in F_{\beta,\epsilon,\Omega}$ for every $\epsilon\in\mathcal{E}$ and $\beta\ge0$. In addition to this, there exist $Z_0,Z_1>0$ with
$$|W_{\beta}(\tau,\epsilon)|\le Z_1Z_0^{\beta}\beta!\frac{\left|\frac{\tau}{\epsilon^r}\right|}{1+\left|\frac{\tau}{\epsilon^r}\right|^{2k}}\exp\left(\sigma r_b(\beta)\left|\frac{\tau}{\epsilon^r}\right|^{k}\right), \quad \beta\ge0,$$
for every $\tau\in\Omega$.
If one writes $T=|T|e^{i\theta}$, we deduce that
\begin{align*}
\left|k\int_{L_{\gamma}}W_{\beta}(u,\epsilon)e^{\left(\frac{u}{T}\right)^{k}}\frac{du}{u}\right|&\le k\int_{0}^{\infty}|W_{\beta}(se^{i\gamma},\epsilon)|e^{-\frac{s^{k}}{|T|^{k}}\cos(k(\gamma-\arg(T)))}ds\\
&\le kZ_1Z_0^{\beta}\beta! \int_{0}^{\infty}\exp(\left[\frac{\sigma\xi(b)}{|\epsilon|^{rk}}-\frac{\Delta}{|T|^k}\right]s^k)ds,
\end{align*}
for every $\beta\ge0$.
This entails that the function $\mathcal{L}_{m_{k}}^{d}(W_{\beta}(\tau,\epsilon))(T)$ is well defined for $T\in S_{d,\theta,|\epsilon|^{r}\left(\frac{\Delta}{\sigma\xi(b)}\right)^{1/k}},$ for every $\frac{\pi}{k}<\theta<\frac{\pi}{k}+2\delta$.
Moreover,
$$(T,z)\mapsto Y(T,z,\epsilon):=\sum_{\beta\ge0}\mathcal{L}_{m_{k}}^{d}(W_{\beta}(\tau,\epsilon))(T)\frac{z^{\beta}}{\beta!}$$
defines a holomorphic function on $S_{d,\theta,|\epsilon|^{r}\left(\frac{\Delta}{\sigma\xi(b)}\right)^{1/k}}\times D(0,\frac{1}{Z_0})$, and it turns out to be a solution of the problem (\ref{e257}), (\ref{e258}) from the properties of Laplace transform in 2), Section~\ref{sec42} and the fact that (\ref{e352a}) is a formal solution of (\ref{e127}), (\ref{e127a}).


\end{proof}

\section{Formal series solutions and strongly regular multi-level asymptotic expansions in a complex parameter for a Cauchy problem}\label{seccionotra2}

In this section we construct analytic solutions of the main problem under study and provide a formal solution in the perturbation parameter linked to the analytic one by means of asymptotic expansions of different nature. This behavior depends on the coefficients appearing in the problem and its geometry. A multi-level asymptotic expansion is observed where not only classical Gevrey asymptotic expansions come up, but also others belonging to a larger class, related to the so called strongly regular sequences.

At this point, we recall the main properties of strongly regular sequences and its related asymptotic expansions. We refer to~\cite{javi1} and~\cite{lamasa} for further details.

\subsection{Strongly regular sequences, asymptotic expansions and summability}\label{subs51}

In the following, $\mathbb{M}=(M_p)_{p\in\N_0}$ stands for a sequence of positive real numbers, with $M_0=1$.

\begin{defin}\label{defi310}
We say $\mathbb{M}$ is a strongly regular sequence if the following conditions hold:
\begin{itemize}
\item[$(\alpha_0)$] $\mathbb{M}$ is logarithmically convex, it is to say, $M_p^2\le M_{p-1}M_{p+1}$ for all positive integer $p$.
\item[$(\mu)$] $\mathbb{M}$ is of moderate growth: there exists $A>0$ such that $M_{p+q}\le A^{p+q}M_{p}M_{q}$ for every $p,q\in\N_0$.
\item[$(\gamma_1)$] $\mathbb{M}$ satisfies the strong non-quasianalyticity condition: there exists $B>0$ such that
$$\sum_{q\ge p}\frac{M_{q}}{(q+1)M_{q+1}}\le B\frac{M_{p}}{M_{p+1}},\quad p\in\N_0.$$
\end{itemize}
\end{defin}

Any Gevrey sequence of certain positive order $\alpha$, $(p!^{\alpha})_{p\ge0}$, is a strongly regular sequence, but there are other sequences appearing in the literature such as those related to the  $1^{+}$ level involved in the study of difference equations (see~\cite{im} among others).

One can generalize the concept of Gevrey asymptotic expansions in the framework of strongly regular sequences.
%
From now on, given a sector $S$ with vertex at the origin, we say $T$ is a (proper) subsector of $S$ whenever the radius of $T$ is finite and $\overline{T}\setminus\{0\}\subseteq S$.

\begin{defin}\label{defi314}
Let $(\mathbb{F},\left\|\cdot\right\|_{\mathbb{F}})$ be a complex Banach space. Given a strongly regular sequence $\mathbb{M}=(M_{p})_{p\ge0}$, we say a holomorphic function $f$, defined in a sector $S\subseteq\C$ with vertex at the origin and with values in $\mathbb{F}$ admits the formal power series $\hat{f}=\sum_{p=0}^{\infty}a_pz^p\in\mathbb{F}[[z]]$ as its $\mathbb{M}$-asymptotic expansion in $S$ (when $z$ tends to 0) if for every (proper) subsector $T$ of $S$ there exist $C_T, A_T>0$ such that for all integer $n\ge1$, one has
$$\left\|f(z)-\sum_{0\le p\le n-1}a_pz^p\right\|_{\mathbb{F}}\le C_TA_T^nM_n|z|^n,\quad z\in T.$$
\end{defin}

Let $\mathbb{M}$ be a sequence of positive real numbers verifying properties $(\alpha_0)$ and $(\gamma_1)$. The map $$h_{\mathbb{M}}(t):=\inf_{p\ge0}M_{p}t^p,\quad t>0; \qquad h_{\mathbb{M}}(0)=0,$$
defines a non-decreasing continuous map in $[0,\infty)$ onto $[0,1]$. Null $\mathbb{M}$-asymptotic expansions are characterized by means of $h_{\mathbb{M}}$ in the following terms.

\begin{prop}\label{prop332}
Let $\mathbb{F}$ be a complex Banach space and let $\mathbb{M}$ be a strongly regular sequence. The following assertions are equivalent:
\begin{itemize}
\item[(i)] $f$ admits null $\mathbb{M}$-asymptotic expansion (i.e.  in Definition~\ref{defi314}, $a_p=0$ for all $p\ge0$) in a sector $S$.
\item[(ii)] For every (proper) subsector $T$ of $S$ there exist $C_T,A_T>0$ such that
\begin{equation}\label{e337}
\left\|f(z)\right\|_{\mathbb{F}}\le C_Th_{\mathbb{M}}\left(A_T|z|\right),
\end{equation}
for all $z\in T$.
\end{itemize}
\end{prop}
\begin{proof}
It is direct from the definition of $h_{\mathbb{M}}$ and $\mathbb{M}$-asymptotic expansions.
\end{proof}

A concept of summability in this general framework was put forward in~\cite{lamasa}. A relevant role is played by the value $\omega(\mathbb{M})$ defined by
$$\omega(\mathbb{M}):=\lim_{r\to\infty}\frac{\log(r)}{\log^{+}(-\log(h_{\mathbb{M}}(1/r)))}.$$
$\omega(\mathbb{M})$ turns out to be a positive real number when departing from a strongly regular sequence $\mathbb{M}$ (see~\cite{javi1}).

\begin{defin}\label{defi344}
Let $\mathbb{F}$ be a complex Banach space and let $\mathbb{M}$ be a strongly regular sequence and $d\in\R$. We say a formal power series $\hat{f}=\sum_{n\ge0}\frac{f_n}{n!}z^n\in\mathbb{F}[[z]]$ is $\mathbb{M}-$summable in direction $d$ if there exists a sector $S$ with bisecting direction $d$ and opening larger than $\omega(\mathbb{M})\pi$, and a holomorphic function $f$ in $S$ with values in $\mathbb{F}$, which admits $\hat{f}$ as its $\mathbb{M}$-asymptotic expansion in~$S$.
\end{defin}
Under the previous conditions, $f$ is unique.

This last definition generalizes the classical one of Gevrey $\kappa$-summable formal power series, for $\kappa>0$, which corresponds to the case $\mathbb{M}=(p!^{1/\kappa})_{p \ge0}$, for which $\omega(\mathbb{M})$ equals $1/\kappa$. Watson's classical lemma guarantees uniqueness of a function admitting in a sector a
$(p!^{1/\kappa})_{p \ge0}-$asymptotic expansion, whenever the opening of the sector is larger than $\pi/\kappa$.

\subsection{Formal series solutions and strongly regular multi-level asymptotics in a complex parameter}

In this subsection we construct the analytic and formal solutions in powers of the perturbation parameter, and relate them by means of different asymptotic behavior, depending on the nature of the coefficients appearing in the equation and the geometry of the problem.

Let $r_1,s_1$ be nonnegative integers, with $s_1>0$. We take $a\in\C^{\star}$ and define $r$ as in (\ref{e111}).

\begin{defin}\label{defin1}
Let $(\mathcal{E}_{i})_{0\le i\le \nu-1}$ be a finite family of open sectors such that $\mathcal{E}_{i}$ has its vertex at the origin and finite radius $r_{\mathcal{E}_{i}}>0$ for all $0\le i\le \nu-1$. We say this family provides a good covering in $\C^{\star}$ if $\mathcal{E}_{i}\cap\mathcal{E}_{i+1}\neq\emptyset$ for $0\le i\le \nu-1$ (where $\mathcal{E}_{\nu}:=\mathcal{E}_{0}$) and $\cup_{0\le i\le \nu-1}\mathcal{E}_{i}=\mathcal{U}\setminus\{0\}$ for some neighborhood of the origin $\mathcal{U}$.
\end{defin}

We assume $r_{\mathcal{E}_{i}}:=r_{\mathcal{E}}$ for every $0\le i\le \nu-1$, for some $r_{\mathcal{E}}>0$.

\begin{defin}\label{defin2}
Let $(\mathcal{E}_{i})_{0\le i\le \nu -1}$ be a good covering in $\C^{\star}$. Let $\mathcal{T}$ be an open sector with vertex at 0 and finite radius, say $r_{\mathcal{T}}>0$ and a family of open sectors
$$S_{d_i,\theta,r_{\mathcal{E}}^{r}r_{\mathcal{T}}}=\left\{t\in\C^{\star}:|t|\le r_{\mathcal{E}}^{r}\cdot r_{\mathcal{T}}, |d_i-\arg(t)|<\frac{\theta}{2}\right\},$$
with $d_i\in[0,2\pi)$ for $0\le i \le \nu-1$, and $\pi/k<\theta<\pi/k+\delta$ for some small enough $\delta>0$, under the following properties:
\begin{enumerate}
\item One has $\arg(d_i)\neq\frac{\pi(2j+1)+\arg(a)}{ks_1}$ for all $j=0,\ldots,ks_1-1$.
\item For every $0\le i \le \nu-1$, $t\in\mathcal{T}$ and $\epsilon\in\mathcal{E}_{i}$ one has $\epsilon^rt\in S_{d_{i},\theta,r_{\mathcal{E}}^rr_{\mathcal{T}}}$.
\end{enumerate}
Under the previous settings, we say the family $((S_{d_{i},\theta,r_{\mathcal{E}}^rr_{\mathcal{T}}})_{0\le i\le \nu-1},\mathcal{T})$ is associated with the good covering $(\mathcal{E}_{i})_{0\le i \le \nu-1}$.
\end{defin}

Let us consider a strongly regular sequence $\mathbb{M}=(M_{p})_{p\ge 0}$, and we assume that $\omega(\mathbb{M})<1/(rk)$ (the case $\omega(\mathbb{M})>1/(rk)$ may be treated along similar lines, while the situation in which $\omega(\mathbb{M})=1/(rk)$ is not included in our study).

Let $(\mathcal{E}_i)_{0\le i \le \nu-1}$ be a good covering which satisfies that the opening of $\mathcal{E}_{i}$ is larger than $\omega(\mathbb{M})\pi$ for every $0\le i\le \nu-1$. We also choose a family $((S_{d_i,\theta,r_{\mathcal{E}}^rr_{\mathcal{T}}})_{0\le i\le \nu-1},\mathcal{T})$ associated with the previous good covering. For the sake of brevity, for every $0\le i\le \nu-1$ we put $S_{d_i}:=S_{d_i,\theta,r_{\mathcal{E}}^rr_{\mathcal{T}}}$ and $\Omega_{i}:=S_{d_i}\cup D(0,\rho_0)$.

Let $S\ge1$ be an integer. We also consider a finite subset $\mathcal{S}$ of $\N^3$. For every $(s,\kappa_0,\kappa_1)\in\mathcal{S}$, let $\hat{b}_{s\kappa_0\kappa_1}(z,\epsilon)$ be a formal power series of the form
$$\hat{b}_{s\kappa_0\kappa_1}(z,\epsilon)=\sum_{\beta\ge0}\hat{b}_{s\kappa_0\kappa_1\beta}(\epsilon)\frac{z^{\beta}}{\beta!}.$$
We assume that the formal power series $\hat{b}_{s\kappa_{0}\kappa_{1}}(z,\epsilon)$ belongs to $\mathcal{O}(D)[[\epsilon]]$,
where $\mathcal{O}(D)$ denotes the Banach space of bounded holomorphic functions on a disc $D$ centered at 0 endowed
with the sup norm. We make the hypothesis that $\hat{b}_{s\kappa_{0},\kappa_{1}}(z,\epsilon)$ is
$\mathbb{M}-$summable on $\mathcal{E}_{i}$, for all $0 \leq i \leq \nu-1$ (see Definition~\ref{defi344}). Moreover, we assume that $\hat{b}_{s\kappa_0\kappa_10}(\epsilon)\equiv0$ for every $(s,\kappa_0,\kappa_1)\in\mathcal{S}$ (this condition will only be needed in the proof of our last result, Theorem \ref{teo726}). We denote by
$\epsilon \mapsto b_{s\kappa_{0}\kappa_{1}}^{(i)}(z,\epsilon)$ the $\mathbb{M}-$sum of $\hat{b}_{s\kappa_{0},\kappa_{1}}(z,\epsilon)$ in $\mathcal{E}_{i}$ (seen as
holomorphic function on $\mathcal{E}_{i}$ with values in $\mathcal{O}(D)$). Using Cauchy formula (in the variable $z$)
together with Proposition~\ref{prop332}, we get
in particular that the coefficients $b_{s\kappa_{0}\kappa_{1}\beta}^{(i)}(\epsilon)$ of the Taylor expansion of
$b_{s\kappa_{0}\kappa_{1}}^{(i)}(z,\epsilon)$,
$$ b_{s\kappa_{0}\kappa_{1}}^{(i)}(z,\epsilon) = \sum_{\beta \geq 0} b_{s\kappa_{0}\kappa_{1}\beta}^{(i)}(\epsilon)
\frac{z^{\beta}}{\beta !} $$
satisfy estimates of the following form: There exist constants $c_1,c_2,c_3>0$ with
\begin{equation}\label{e443}
|b_{s\kappa_{0}\kappa_{1}\beta}^{(i)}(\epsilon)| \leq c_1c_2^{\beta} \beta!, \quad\epsilon\in\mathcal{E}_i,
\end{equation}
and
\begin{equation}\label{e444}
|b_{s\kappa_{0}\kappa_{1}\beta}^{(i+1)}(\epsilon) - b_{s\kappa_{0}\kappa_{1}\beta}^{(i)}(\epsilon)| \leq c_1c_2^{\beta} \beta!
h_{\mathbb{M}}(c_3|\epsilon|),
\quad\epsilon\in\mathcal{E}_i\cap\mathcal{E}_{i+1},
\end{equation}
for all $\beta \geq 0$.




We study the formal problem
\begin{equation*}
(\epsilon^{r_1}(t^{k+1}\partial_t)^{s_1}+a)\partial_{z}^{S}\hat{X}(t,z,\epsilon)= \sum_{(s,\kappa_0,\kappa_1)\in\mathcal{S}}\hat{b}_{s\kappa_0\kappa_1}(z,\epsilon) t^s(\partial_t^{\kappa_0}\partial_z^{\kappa_1}\hat{X})(t,z,\epsilon),
\end{equation*}
and for every $0\le i\le \nu-1$, we study the Cauchy problem
\begin{equation}\label{e394}
(\epsilon^{r_1}(t^{k+1}\partial_t)^{s_1}+a)\partial_{z}^{S}X_{i}(t,z,\epsilon)=\sum_{(s,\kappa_0,\kappa_1)\in\mathcal{S}}b^{(i)}_{s\kappa_0\kappa_1}(z,\epsilon)t^s(\partial_t^{\kappa_0}\partial_z^{\kappa_1}X_i)(t,z,\epsilon)
\end{equation}
under the initial conditions
\begin{equation}\label{e395}
(\partial_z^{j}(X_i))(t,0,\epsilon)=\phi_{i,j}(t,\epsilon),\quad 0\le j\le S-1,
\end{equation}
where the functions $\phi_{i,j}$ are constructed in the following manner.
For $0\le i\le \nu-1$ and $0\le j\le S-1$, we consider a function defined in $\Omega_i\times\mathcal{E}_i$, $(\tau,\epsilon)\mapsto W_{i,j}(\tau,\epsilon)$, such that:
\begin{enumerate}
\item[a)] The function $W_{i,j}$ is holomorphic in $\Omega_i\times \mathcal{E}_{i}$.
\item[b)] For every $\epsilon\in\mathcal{E}_i$, the function $\tau\mapsto W_{i,j}(\tau,\epsilon)$ belongs to the space $F_{j,\epsilon,\Omega_i}$, and
\begin{equation}\label{e325}
\sup_{\epsilon\in\mathcal{E}_i}\left\|W_{i,j}(\tau,\epsilon)\right\|_{j,\epsilon,\Omega_i}<\infty.
\end{equation}
\item[c)]  There exists $K>0$ such that
\begin{equation}\label{e442}
\sup_{\epsilon\in\mathcal{E}_{i}\cap\mathcal{E}_{i+1}}\frac{\left\|W_{i+1,j}(\tau,\epsilon)- W_{i,j}(\tau,\epsilon)\right\|_{j,\epsilon,\Omega_{i}\cap\Omega_{i+1}}}{h_{\mathbb{M}}(K|\epsilon|)} <\infty.
\end{equation}
\end{enumerate}
Let $\gamma_i\in[0,2\pi)$ such that $L_{\gamma_i}:=\mathbb{R}_{+}e^{\gamma_i\sqrt{-1}}\subseteq S_{d_i}\cup\{0\}$. We define
$$Y_{i,j}(T,\epsilon):=\mathcal{L}_{m_{k}}^{d_i}(W_{i,j}(\tau,\epsilon))(T).$$
This definition makes sense from the growth bounds on the function $\tau\mapsto W_{i,j}(\tau,\epsilon)$  given in (\ref{e325}), as it can be checked from Definition~\ref{defi74} and Definition~\ref{defi2}. The function $T\mapsto Y_{i,j}(T,\epsilon)$ turns out to be holomorphic for all $T=|T|e^{i\theta}$ such that $\cos(k(\gamma_i-\theta))\ge \Delta$, for some $\Delta>0$, and $|T|\le |\epsilon|^{r}\frac{\Delta^{1/k}}{(\sigma \xi(b))^{1/k}}$, where $\xi(b)=\sum_{n\ge0}\frac{1}{(n+1)^b}$. Indeed, bearing in mind Theorem~\ref{teo280}, the function $(T,z)\mapsto Y_{i}(T,z,\epsilon)$ defined as in (\ref{e282}) by
$$
Y_i(T,z,\epsilon)=\sum_{j\ge 0}Y_{i,j}(T,\epsilon)\frac{z^{j}}{j!}
$$
is a solution of (\ref{e257}), (\ref{e258}) in $S_{d_i,\theta,|\epsilon|^{r}\left(\frac{\Delta}{\sigma \xi(b)}\right)^{1/k}}\times D(0,\frac{1}{Z_0})$.

Finally, we put
$$\phi_{i,j}(t,\epsilon)=Y_{i,j}(\epsilon^rt,\epsilon):=k\int_{L_{\gamma_{i}}}W_{i,j}(u,\epsilon)e^{-\left(\frac{u}{\epsilon^r t}\right)^{k}}\frac{du}{u},$$
for every $(t,\epsilon)\in \mathcal{T}\times \mathcal{E}_{i}$. From b), $\phi_{i,j}$ is well defined and from a), $\phi_{i,j}(t,\epsilon)$ turns out to be a holomorphic function in $\mathcal{T}\times \mathcal{E}_{i}$.

We are in a position to construct the solution of the main problem.

\begin{theo}\label{teosol}
Let the initial data (\ref{e395}) be constructed as before. Under assumptions (A) and (B), the problem (\ref{e394}), (\ref{e395}) has a holomorphic and bounded solution $X_{i}(t,z,\epsilon)$ on $(\mathcal{T}\cap D(0,h'))\times D(0,R_0)\times \mathcal{E}_{i}$, for every $0\le i \le \nu-1$, for some $R_0,h'>0$.
\end{theo}
\begin{proof}
Regarding the construction of the initial conditions and Theorem~\ref{teo280}, one has that $(T,z)\mapsto Y_{i}(T,z,\epsilon)$ is a solution of (\ref{e257}), (\ref{e258}) in $S_{d_i,\theta,|\epsilon|^{r}\left(\frac{\Delta}{\sigma \xi(b)}\right)^{1/k}}\times D(0,\frac{1}{Z_0})$. One concludes the result by defining
\begin{equation}\label{e437}
X_{i}(t,z,\epsilon):=Y_i(\epsilon^r t,z,\epsilon)=\sum_{\beta\ge 0}k\int_{L_{\gamma_{i}}}W_{i,\beta}(u,\epsilon)e^{-\left(\frac{u}{t\epsilon^r}\right)^{k}}\frac{du}{u}\frac{z^{\beta}}{\beta!},
\end{equation}
and taking into account the properties of Laplace transform described in Proposition~\ref{prop3}.
The elements of the sequence $(W_{i,\beta})_{\beta\ge0}$ are the coefficients of (\ref{e139}), determined by the recursion (\ref{e145}), formal solution of the problem (\ref{e127}), (\ref{e128}).
\end{proof}

Next we study the rate of growth of the differences between the coefficients of like powers in the formal solution obtained in Proposition~\ref{prop150} for the problem (\ref{e127}), (\ref{e128}), when considered in two adjacent sectors of the good covering.

\begin{prop}\label{prop200}
Let $0\le i\le \nu-1$. In the situation described in the present section, and under Assumptions (A) and (B), 
there exist $\tilde{c}_0,\tilde{c}_1,\tilde{K}_2>0$ such that
for every $\beta\ge0$ and $\epsilon\in\mathcal{E}_{i}\cap\mathcal{E}_{i+1}$ one has
\begin{equation}\label{e494}
\left\|W_{i+1,\beta}(\tau,\epsilon)- W_{i,\beta}(\tau,\epsilon)\right\|_{\beta,\epsilon,\Omega_{i}\cap\Omega_{i+1}}\le \tilde{c}_1\tilde{c}_0^\beta \beta! h_{\mathbb{M}}\left(\tilde{K}_2|\epsilon|\right).
\end{equation}
\end{prop}

\begin{proof}
For every $\beta\ge 0$, the recursion formula (\ref{e145}) allows us to write  $\frac{W_{i+1,\beta+S}(\tau,\epsilon)-W_{i,\beta+S}(\tau,\epsilon)}{\beta!}$ in the following way:
\begin{align*}
&\frac{1}{(k\tau^k)^{s_1}+a}\sum_{(s,\kappa_0,\kappa_1)\in\mathcal{S}} \sum_{\alpha_0+\alpha_1=\beta} \left(\frac{b^{(i+1)}_{s\kappa_0\kappa_1\alpha_0}(\epsilon) -b^{(i)}_{s\kappa_0\kappa_1\alpha_0}(\epsilon)}{\alpha_0!} \epsilon^{-r(s-\kappa_0)}\right.\\
&\times \left[\frac{\tau^{k}}{\Gamma\left(\frac{\delta_{\kappa_0}}{k}\right)} \int_{0}^{\tau^k}(\tau^k-\sigma)^{\frac{\delta_{\kappa_0}}{k}-1} (k\sigma)^{\kappa_0}\frac{W_{i+1,\alpha_1+\kappa_1}(\sigma^{1/k},\epsilon)}{\alpha_1!} \frac{d\sigma}{\sigma} \right.\\
&\left.+\sum_{1\le p\le \kappa_0-1}A_{\kappa_0,p}\frac{\tau^k}{\Gamma\left(\frac{\delta_{\kappa_0}+k(\kappa_0-p)}{k} \right)}\int_{0}^{\tau^k}(\tau^k-\sigma)^{\frac{\delta_{\kappa_0}+k(\kappa_0-p)}{k}-1}(k\sigma)^{p}\frac{W_{i+1,\alpha_1+\kappa_1}(\sigma^{1/k},\epsilon)}{\alpha_1!}\frac{d\sigma}{\sigma}\right]\\
&+\left[\frac{\tau^{k}}{\Gamma\left(\frac{\delta_{\kappa_0}}{k}\right)} \int_{0}^{\tau^k}(\tau^k-\sigma)^{\frac{\delta_{\kappa_0}}{k}-1} (k\sigma)^{\kappa_0}\frac{(W_{i+1,\alpha_1+\kappa_1}(\sigma^{1/k},\epsilon)- W_{i,\alpha_1+\kappa_1}(\sigma^{1/k},\epsilon))}{\alpha_1!}\frac{d\sigma}{\sigma} \right.\\
& + \sum_{1\le p\le \kappa_0-1}A_{\kappa_0,p} \frac{\tau^k}{\Gamma\left(\frac{\delta_{\kappa_0}+k(\kappa_0-p)}{k}\right)} \int_{0}^{\tau^k}(\tau^k-\sigma)^{\frac{\delta_{\kappa_0}+k(\kappa_0-p)}{k}-1} (k\sigma)^{p}\\
&\left.\left.\frac{(W_{i+1,\alpha_1+\kappa_1}(\sigma^{1/k},\epsilon)- W_{i,\alpha_1+\kappa_1}(\sigma^{1/k},\epsilon))}{\alpha_1!}\frac{d\sigma}{\sigma}\right] \times\frac{b^{(i)}_{s\kappa_0\kappa_1\alpha_0}(\epsilon)}{\alpha_0!}\epsilon^{-r(s-\kappa_0)}\right),
\end{align*}
for $\epsilon\in\mathcal{E}:=\mathcal{E}_{i}\cap\mathcal{E}_{i+1}$ and $\tau\in\Omega:=\Omega_i\cap\Omega_{i+1}$.

Let $\epsilon\in\mathcal{E}$. Taking norms $\left\|\cdot\right\|_{\beta+S,\epsilon,\Omega}$ in the previous expression and applying Lemma~\ref{lema1}, Lemma~\ref{lema2} and Proposition~\ref{prop1}, one arrives at
$$\frac{\left\|W_{i+1,\beta+S}(\tau,\epsilon)- W_{i,\beta+S}(\tau,\epsilon)\right\|_{\beta+S,\epsilon,\Omega}}{\beta!}$$
\begin{align*}
&\le C_1\sum_{(s,\kappa_0,\kappa_1)\in\mathcal{S}}\sum_{\alpha_0+\alpha_1=\beta} \left(\frac{|b_{s\kappa_0\kappa_1\alpha_0}^{(i+1)}(\epsilon)- b_{s\kappa_0\kappa_1\alpha_0}^{(i)}(\epsilon)|}{\alpha_0!} |\epsilon|^{-r(s-\kappa_0)} \frac{\left\|W_{i+1,\alpha_1+\kappa_1}(s^{1/k},\epsilon)\right\|_{\alpha_1+\kappa_1,\epsilon,\Omega}}{\alpha_1!}
\right.\\
&\times |\epsilon|^{rk(\frac{\delta_{\kappa_0}}{k}+\kappa_0)} \left(\frac{(\beta+S+1)^b}{\beta+S-\alpha_1-\kappa_1}\right)^{\frac{\delta_{\kappa_0}}{k}+\kappa_0+1} \left[ \frac{C_2k^{\kappa_0}}{\Gamma\left(\frac{\delta_{\kappa_0}}{k}\right)}+ \sum_{1\le p\le \kappa_0-1}|A_{\kappa_0,p}| \frac{C_2k^p}{\Gamma\left(\frac{\delta_{\kappa_0}+k(\kappa_0-p)}{k}\right)}\right]\\
&+|\epsilon|^{rk(\frac{\delta_{\kappa_0}}{k}+\kappa_0)}\times \frac{\left\|W_{i+1,\alpha_1+\kappa_1}(s^{1/k},\epsilon)- W_{i,\alpha_1+\kappa_1}(s^{1/k},\epsilon)\right\|_{\alpha_1+\kappa_1,\epsilon,\Omega}}{\alpha_1!}\\
&\left.\times \left[ \frac{C_2k^{\kappa_0}}{\Gamma\left(\frac{\delta_{\kappa_0}}{k}\right)}+ \sum_{1\le p \le \kappa_0-1}|A_{\kappa_0,p}| \frac{C_2k^p}{\Gamma\left(\frac{\delta_{\kappa_0}+k(\kappa_0-p)}{k}\right)}\right] \left(\frac{(\beta+S+1)^b}{\beta+S-\alpha_1-\kappa_1}\right)^{\frac{\delta_{\kappa_0}}{k}+\kappa_0+1} \frac{|b^{(i)}_{s\kappa_0\kappa_1\alpha_0}(\epsilon)|}{\alpha_0!} |\epsilon|^{-r(s-\kappa_0)}\right)
\end{align*}
for suitable $C_1,C_2>0$.
For the sake of brevity, we put
$$D_{\kappa_0}:=\frac{k^{\kappa_0}}{\Gamma\left(\frac{\delta_{\kappa_0}}{k} \right)}+ \sum_{1\le p \le \kappa_0-1}|A_{\kappa_0,p}| \frac{k^p}{\Gamma\left(\frac{\delta_{\kappa_0}+k(\kappa_0-p)}{k}\right)},$$
and $\eta_{\kappa_0}:=\lfloor b(\frac{\delta_{\kappa_0}}{k}+\kappa_0+1)\rfloor-1$. Analogous estimates as in the proof of Proposition~\ref{prop150} yield
$$\frac{\left\|W_{i+1,\beta+S}(\tau,\epsilon)- W_{i,\beta+S}(\tau,\epsilon)\right\|_{\beta+S,\epsilon,\Omega}}{\beta!}$$
\begin{align*}
&\le C_1C_2C_3C_4\sum_{(s,\kappa_0\kappa_1)\in\mathcal{S}} \sum_{\alpha_0+\alpha_1=\beta}D_{\kappa_0}\frac{\beta!}{(\beta-\eta_{\kappa_0})!}\\ &\times\left[\frac{|b_{s\kappa_0\kappa_1\alpha_0}^{(i+1)}(\epsilon)- b_{s\kappa_0\kappa_1\alpha_0}^{(i)}(\epsilon)|}{\alpha_0!} \frac{\left\| W_{i+1,\alpha_1+\kappa_1}(s^{1/k},\epsilon) \right\|_{\alpha_1+\kappa_1,\epsilon,\Omega}}{\alpha_1!} \right.\\
&\left.+ \frac{\left\|W_{i+1,\alpha_1+\kappa_1}(s^{1/k},\epsilon)- W_{i,\alpha_1+\kappa_1}(s^{1/k},\epsilon) \right\|_{\alpha_1+\kappa_1,\epsilon,\Omega}}{\alpha_1!} \frac{|b^{(i)}_{s\kappa_0\kappa_1\alpha_0}(\epsilon)|}{\alpha_0!}    \right],
\end{align*}
for some $C_3,C_4>0$.

We define $w_{\beta}(\epsilon):=\left\|W_{i+1,\beta}(\tau,\epsilon)- W_{i,\beta}(\tau,\epsilon)\right\|_{\beta,\epsilon,\Omega}$ for every $\beta\ge 0$. By taking into account (\ref{e194}), (\ref{e443}) and (\ref{e444}) the previous expression allows us to write
\begin{align*}
\frac{w_{\beta+S}(\epsilon)}{\beta!}&\le
C_1C_2C_3C_4\sum_{(s,\kappa_0\kappa_1)\in\mathcal{S}} \sum_{\alpha_0+\alpha_1=\beta}D_{\kappa_0}\frac{\beta!}{(\beta-\eta_{\kappa_0})!}\\ &\times\left[c_1c_2^{\alpha_0}h_{\mathbb{M}}(c_3|\epsilon|) \frac{Z_1Z_0^{\alpha_1+\kappa_1}(\alpha_1+\kappa_1)!}{\alpha_1!}
+ \frac{w_{\alpha_1+\kappa_1}(\epsilon)}{\alpha_1!} c_1c_2^{\alpha_0}    \right].
\end{align*}
Finally, let us put for every $\beta\ge 0$
$$
v_{\beta}:=\sup_{\epsilon\in\mathcal{E}_{i}\cap\mathcal{E}_{i+1}}
\frac{w_{\beta}(\epsilon)}{h_{\mathbb{M}}\left(\tilde{K}_2|\epsilon|\right)},
$$
where $\tilde{K}_2>0$ is the maximum of the constants $K$ and $c_3$ appearing in the estimates (\ref{e442}) and (\ref{e444}), respectively. Taking into account that $t\mapsto h_{\mathbb{M}}(t)$ is increasing, the previous inequalities imply that $v_{\beta}$ is finite for every $\beta$, and moreover
\begin{align}\label{e465}
\frac{v_{\beta+S}}{\beta!}&\le
C_1C_2C_3C_4\sum_{(s,\kappa_0\kappa_1)\in\mathcal{S}} \sum_{\alpha_0+\alpha_1=\beta}D_{\kappa_0}\frac{\beta!}{(\beta-\eta_{\kappa_0})!} c_1c_2^{\alpha_0}
\left[ \frac{Z_1Z_0^{\alpha_1+\kappa_1}(\alpha_1+\kappa_1)!}{\alpha_1!} + \frac{v_{\alpha_1+\kappa_1}}{\alpha_1!}     \right].
\end{align}
Let $$b(x)=\frac{c_1}{1-c_2x}=c_1\sum_{\alpha\ge0}c_2^{\alpha}x^{\alpha},\qquad
c(x)=\frac{Z_1}{1-Z_0x}=Z_1\sum_{\alpha\ge0}Z_0^{\alpha}x^{\alpha},$$
which are both holomorphic functions in some neighborhood of the origin. We consider the Cauchy problem
\begin{equation}\label{e470}
u^{(S)}(x)=C_1C_2C_3C_4\sum_{(s,\kappa_0,\kappa_1)\in\mathcal{S}}
D_{\kappa_0}x^{\eta_{\kappa_0}}\left( b(x)\big(u(x)+c(x)\big)^{(\kappa_1)}\right)^{(\eta_{\kappa_0})},
\end{equation}
with initial conditions
\begin{equation}\label{e471}
u^{(j)}(0)=v_{j},\quad 0\le j\le S-1.
\end{equation}
One may check that the problem (\ref{e470}), (\ref{e471}) has a unique formal solution
$$u(x)=\sum_{\beta\ge0}u_{\beta}\frac{x^\beta}{\beta!}\in\mathbb{R}[[x]],$$
whose coefficients satisfy the recursion formula
\begin{equation}\label{e480}
\frac{u_{\beta+S}}{\beta!}=
C_1C_2C_3C_4\sum_{(s,\kappa_0\kappa_1)\in\mathcal{S}} \sum_{\alpha_0+\alpha_1=\beta}D_{\kappa_0}\frac{\beta!}{(\beta-\eta_{\kappa_0})!} c_1c_2^{\alpha_0}\left[ \frac{Z_1Z_0^{\alpha_1+\kappa_1}(\alpha_1+\kappa_1)!}{\alpha_1!} + \frac{u_{\alpha_1+\kappa_1}}{\alpha_1!}     \right].
\end{equation}
for every $\beta\ge0$. It is clear from the initial conditions that $u_j=v_j$ for $0\le j \le S-1$, and one easily concludes from (\ref{e465}) and (\ref{e480}) that $v_{\beta}\le u_{\beta}$ for every $\beta\ge0$.
The classical theory of existence of solutions of ODEs can be applied to the problem (\ref{e470}), (\ref{e471}) to guarantee that the formal power series $u(x)$ is convergent in a neighborhood of the origin, so there exist $\tilde{c}_0,\tilde{c}_1>0$ such that
$u_{\beta}<\tilde{c}_1\tilde{c}_0^\beta \beta!$ for every $\beta\ge0$, and hence
\begin{equation*}
\left\|W_{i+1,\beta}(\tau,\epsilon)- W_{i,\beta}(\tau,\epsilon)\right\|_{\beta,\epsilon,\Omega}\le \tilde{c}_1\tilde{c}_0^\beta \beta! h_{\mathbb{M}}\left(\tilde{K}_2|\epsilon|\right),
\end{equation*}
for all $\epsilon\in\mathcal{E}$ and $\tau\in\Omega$, as desired.
\end{proof}

The next two results, crucial on determining the asymptotic behavior of the solution, provide the key to apply a novel version of the so called Ramis-Sibuya theorem. The geometry of the problem gives rise to different asymptotic behaviors we are about to deal with in different situations.

The first case to consider is that in which no singularities fall in between two different integration directions $L_{\gamma_i}$ and $L_{\gamma_{i+1}}$ for some $0\le i\le \nu-1$, and they can be chosen to coincide. At this point, the difference of two actual solutions is an $\mathbb{M}$-flat function.


\begin{theo}\label{teo445}
Let $0\le i\le \nu-1$. In the situation described in the present section, we make Assumptions (A) and (B), and also assume that there are no singular directions $\frac{\pi(2j+1)+\arg(a)}{ks_1}$ for $j=0,\ldots,ks_1-1$ in between $\gamma_i$ and $\gamma_{i+1}$.  Then, there exist $\tilde{K}_1,\tilde{K}_2>0$ such that
\begin{equation*}
|X_{i+1}(t,z,\epsilon)-X_{i}(t,z,\epsilon)|\le \tilde{K}_1h_{\mathbb{M}}\left(\tilde{K}_2|\epsilon|\right),
\end{equation*}
for every $t\in\mathcal{T}\cap D(0,h')$, $\epsilon\in\mathcal{E}_{i}\cap\mathcal{E}_{i+1}$ and all $z\in D(0,\tilde{R}_0)$, for some $\tilde{R}_0>0$.
\end{theo}
\begin{proof}
As a consequence of (\ref{e494}), for every $\beta\ge 0$ one has
\begin{equation*}
|W_{i+1,\beta}(\tau,\epsilon)-W_{i,\beta}(\tau,\epsilon)|\le \tilde{c}_1\tilde{c}_0^{\beta}\beta!h_{\mathbb{M}}\left(\tilde{K}_2|\epsilon|\right) \frac{\left|\frac{\tau}{\epsilon^r}\right|}{1+\left|\frac{\tau}{\epsilon^r}\right|^{2k}} \exp\left(\sigma r_{b}(\beta+S)\left|\frac{\tau}{\epsilon^r}\right|^k\right)
\end{equation*}
for all $\epsilon\in\mathcal{E}$, $\tau\in\Omega$ (we keep the notation in the previous proof).
Under the hypothesis made on the geometry of the problem, $L_{\gamma_i}$ and $L_{\gamma_{i+1}}$ can be chosen in such a way that $L_{\gamma_i}\equiv L_{\gamma_{i+1}}$. At this point, one can estimate the difference of $X_{i}$ and $X_{i+1}$ regarding their definition in (\ref{e437}):
\begin{equation}\label{e512}
\left|X_{i+1}(t,z,\epsilon)-X_{i}(t,z,\epsilon) \right|
\end{equation}
\begin{align*}
&\le k\sum_{\beta\ge0}\int_{0}^{\infty}|W_{i+1,\beta}(|u|e^{\sqrt{-1}\gamma_{i}},\epsilon) -W_{i,\beta}(|u|e^{\sqrt{-1}\gamma_{i}},\epsilon)| e^{-\frac{|u|^{k}}{|t|^{k}|\epsilon|^{rk}} \cos\left(k(\gamma_{i}-\arg(t\epsilon^r))\right)} \frac{d|u|}{|u|}\frac{|z|^{\beta}}{\beta!}\\
&\le k\tilde{c}_1\sum_{\beta\ge0}(\tilde{c}_0|z|)^{\beta} h_{\mathbb{M}}\left(\tilde{K}_2|\epsilon|\right) \int_{0}^{\infty}\frac{\left|\frac{u}{\epsilon^r}\right|}{1+\left|\frac{u}{\epsilon^r}\right|^{2k}} \exp\left((\sigma\xi(b)-\cos(k(\gamma_{i}-\arg(t\epsilon^r)))r_{\mathcal{T}}^{-k}) \frac{|u|^{k}}{|\epsilon|^{rk}}\right)\frac{d|u|}{|u|}\\
&\le k\tilde{c}_1h_{\mathbb{M}}\left(\tilde{K}_2|\epsilon|\right)|\epsilon|^{rk} \int_{0}^{\infty}\exp\left(-C_{5}\frac{|u|^{k}}{|\epsilon|^{rk}}\right) d\left(\frac{|u|}{|\epsilon|^{rk}}\right)\sum_{\beta\ge0}(\tilde{c}_0|z|)^{\beta}\\
&\le \tilde{K}_1h_{\mathbb{M}}\left(\tilde{K}_2|\epsilon|\right)
\end{align*}
for suitable constants $C_5,\tilde{K}_1>0$, as long as $t\in\mathcal{T}\cap D(0,h')$, $\epsilon\in\mathcal{E}$ and $z\in D(0,1/(2\tilde{c}_0))$.
\end{proof}

The second situation considers the case when a singularity lies in between two directions of integration, $L_{\gamma_{i}}$ and $L_{\gamma_{i+1}}$. This situation is handled by a deformation path argument.


\begin{theo}\label{teo526}
Let $0\le i\le \nu-1$. In the situation described in the present section, we make Assumptions (A) and (B), and moreover assume now that there exists $j\in\{0,\ldots,ks_1-1\}$ such that the singular direction $\frac{\pi(2j+1)+\arg(a)}{k s_1}$ lies in between $\gamma_i$ and $\gamma_{i+1}$. Then, there exist $\tilde{K}_3,\tilde{K}_4>0$ such that
\begin{align*}
|X_{i+1}(t,z,\epsilon)-X_{i}(t,z,\epsilon)|\le \tilde{K}_{3}\exp\left(-\frac{\tilde{K}_4}{|\epsilon|^{rk}}\right),
\end{align*}
for every $\epsilon\in\mathcal{E}_{i}\cap\mathcal{E}_{i+1}$, $t\in\mathcal{T}\cap D(0,h')$, and all $z\in D(0,R'_0)$, for some $R'_0>0$.
\end{theo}
\begin{proof}
The definition of the solutions of the problem yields
$$|X_{i+1}(t,z,\epsilon)-X_{i}(t,z,\epsilon)|$$
\begin{equation*}
\le k\sum_{\beta\ge0}\left|\int_{L_{\gamma_{i+1}}}W_{i+1,\beta}(u,\epsilon)e^{-\left(\frac{u}{\epsilon^rt}\right)^k}\frac{du}{u} -\int_{L_{\gamma_{i}}}W_{i,\beta}(u,\epsilon)e^{-\left(\frac{u}{\epsilon^rt}\right)^k}\frac{du}{u}\right|\frac{|z|^{\beta}}{\beta!},
\end{equation*}
for all $\epsilon\in\mathcal{E}_{i}\cap \mathcal{E}_{i+1}$, $t\in\mathcal{T}\cap D(0,h')$ and $z\in D(0,R_0)$.

One may deform the integration path in the expression above in order to get that
\begin{equation*}
|X_{i+1}(t,z,\epsilon)-X_{i}(t,z,\epsilon)|\le k\sum_{\beta\ge0}|I_1-I_2+I_3-I_4+I_5|\frac{|z|^{\beta}}{\beta!},
\end{equation*}
where
\begin{align*}
 I_1&:= \int_{L_{\gamma_{i+1},\rho/2}} W_{i+1,\beta}(u,\epsilon) e^{-(\frac{u}{\epsilon^{r}t})^k} \frac{du}{u},\qquad
I_2:= \int_{L_{\gamma_{i},\rho/2}} W_{i,\beta}(u,\epsilon) e^{-(\frac{u}{\epsilon^{r}t})^k} \frac{du}{u},\\
I_3&:= \int_{C_{\rho/2,\theta_{i,i+1},\gamma_{i+1}}} W_{i+1,\beta}(u,\epsilon) e^{-(\frac{u}{\epsilon^{r}t})^k} \frac{du}{u},\qquad
I_4:= \int_{C_{\rho/2,\theta_{i,i+1},\gamma_{i}}} W_{i,\beta}(u,\epsilon) e^{-(\frac{u}{\epsilon^{r}t})^k} \frac{du}{u},\\
I_5&:= \int_{L_{0,\rho/2,\theta_{i,i+1}}} (W_{i+1,\beta}(u,\epsilon) - W_{i,\beta}(u,\epsilon))
e^{-(\frac{u}{\epsilon^{r}t})^k} \frac{du}{u}.
\end{align*}
Here, $L_{\gamma_{i+1},\rho/2} = [\rho/2, +\infty)e^{\sqrt{-1}\gamma_{i+1}}$,
$L_{\gamma_{i},\rho/2} = [\rho/2, +\infty)e^{\sqrt{-1}\gamma_{i}}$,
$C_{\rho/2,\theta_{i,i+1},\gamma_{i+1}}$ is an arc of circle connecting $(\rho/2)e^{\sqrt{-1}\theta_{i,i+1}}$ and
$(\rho/2)e^{\sqrt{-1}\gamma_{i+1}}$ with a well chosen orientation, where $\theta_{i,i+1}$ equals $(\gamma_{i}+\gamma_{i+1})/2$, and
$C_{\rho/2,\theta_{i,i+1},\gamma_{i}}$ is an arc of circle connecting $(\rho/2)e^{\sqrt{-1}\theta_{i,i+1}}$ and
$(\rho/2)e^{\sqrt{-1}\gamma_{i}}$ with a well chosen orientation. Finally,
$L_{0,\rho/2,\theta_{i,i+1}} = [0,\rho/2]e^{\sqrt{-1}\theta_{i,i+1}}$.

We first settle upper bounds for $|I_1|$. Bearing in mind the estimates in (\ref{e194}) we have that
\begin{equation}\label{e551}
|W_{i,\beta}(\tau,\epsilon)|\le Z_1Z_0^{\beta}\beta!\frac{\left|\frac{\tau}{\epsilon^r}\right|}{1+\left|\frac{\tau}{\epsilon^r}\right|^{2k}}\exp\left(\sigma r_{b}(\beta)\left|\frac{\tau}{\epsilon^r}\right|^{k}\right),
\end{equation}
for every $\beta\ge0$, $\tau\in\Omega$ and $\epsilon\in\mathcal{E}$ (with the previously fixed notation). This last inequality yields
\begin{align*}
|I_1|&\le\int_{\rho/2}^{\infty}|W_{i+1,\beta}(|u|e^{\sqrt{-1}\gamma_{i+1}},\epsilon)| \exp\left(-\frac{|u|^{k}}{|t|^{k}|\epsilon|^{rk}}\cos(k(\gamma_{i+1}-\arg(t\epsilon^r))) \right)\frac{d|u|}{|u|}\\
&\le Z_1Z_0^\beta\beta!\int_{\rho/2}^{\infty}\frac{\left|\frac{u}{\epsilon^r}\right|}{1+\left|\frac{u}{\epsilon^r}\right|^{2k}}\exp\left(\sigma r_{b}(\beta)\left|\frac{u}{\epsilon^r}\right|^{k}\right)\exp\left(-\frac{|u|^{k}}{|t|^{k}|\epsilon|^{rk}}\cos(k(\gamma_{i+1}-\arg(t\epsilon^r))) \right)\frac{d|u|}{|u|}.
\end{align*}
Analogous estimates as those appearing in (\ref{e512}) yield the existence of positive constants $C_6,C_7$ such that
\begin{equation}\label{e821}
|I_1|\le C_6Z_0^\beta \beta!\exp\left(-\frac{C_7}{|\epsilon|^{rk}}\right),
\end{equation}
for all $\beta\ge0$, $\epsilon\in\mathcal{E}$ and $t\in\mathcal{T}\cap D(0,h')$.

One can follow the same steps as before to achieve estimates in the form of (\ref{e821}) when dealing with $|I_2|$. We now give upper bounds for $|I_3|$.

From the fact that (\ref{e551}) holds, one gets
\begin{align*}
|I_3|&\le Z_1Z_0^{\beta}\beta!\frac{\left|\frac{\rho}{2\epsilon^r}\right|}{1+\left|\frac{\rho}{2\epsilon^r}\right|^{2k}}\exp\left(\sigma r_{b}(\beta)\left|\frac{\rho}{2\epsilon^r}\right|^k\right)\frac{2}{\rho}\int_{\theta_{i,i+1}}^{\gamma_{i+1}}\exp\left(-\frac{(\rho/2)^{k}}{|t|^{k}|\epsilon|^{rk}}\cos(k(\theta-\arg(t\epsilon^r)))\right)d\theta\\
&\le \frac{2(\gamma_{i+1}-\theta_{i,i+1})}{\rho}Z_1Z_0^\beta\beta!\exp\left(-C_{8}\frac{(\rho/2)^{k}}{|\epsilon|^{rk}}\right),
\end{align*}
for some $C_8>0$ and every $t\in\mathcal{T}\cap D(0,h')$, $\epsilon\in\mathcal{E}$. This leads to the existence of constants $C_9,C_{10}>0$ with
\begin{equation}\label{e834}
|I_3|\le C_9Z_0^\beta\beta!\exp\left(-\frac{C_{10}}{|\epsilon|^{rk}}\right),
\end{equation}
for all $\beta\ge0$ and $\epsilon\in\mathcal{E}$. The previous estimates may apply to $|I_4|$ accordingly.

We conclude with the study of $|I_5|$. Regarding (\ref{e494}), one has
$$|I_5|\le \int_{0}^{\rho/2}\frac{\tilde{c}_1\tilde{c}_0^{\beta}\beta! h_{\mathbb{M}}(\tilde{K}_2|\epsilon|)\left|\frac{u}{\epsilon^r}\right|}{1+\left|\frac{u}{\epsilon^r}\right|^{2k}}\exp\left(\sigma r_b(\beta)\left|\frac{u}{\epsilon^r}\right|^{k}\right)\exp\left(-\frac{|u|^{k}}{|\epsilon|^{rk}|t|^k}\cos(k(\theta_{i,i+1}-\arg(t\epsilon^r)))\right)\frac{d|u|}{|u|},$$
for some $\tilde{c}_{0},\tilde{c}_{1},\tilde{K}_2>0$, and all $\epsilon\in\mathcal{E}$. If $t\in\mathcal{T}\cap D(0,h')$, then the previous expression can be estimated from above, following the same steps as in (\ref{e512}), by
$$\tilde{c}_1\tilde{c}_0^{\beta}\beta!h_{\mathbb{M}}(\tilde{K}_2|\epsilon|) \int_{0}^{\rho/2}\exp\left(-C_{11}\frac{|u|^{k}}{|\epsilon|^{rk}}\right) d\left(\frac{|u|}{|\epsilon|^{rk}}\right),$$
for some $C_{11}>0$. This entails the existence of $\tilde{c}_2>0$ such that
\begin{equation}\label{e844}
|I_5|\le \tilde{c}_2\tilde{c}_0^{\beta}\beta!h_{\mathbb{M}}(\tilde{K}_2|\epsilon|),
\end{equation}
which is valid for all $\beta\ge0$, and every $\epsilon\in\mathcal{E}$.

At this point, we observe that, as it can be found in~\cite{jj}, one has
$$\omega(\mathbb{M})=\lim_{n\to\infty}\frac{\log(M_{n+1})-\log(M_{n})}{\log(n)}.$$
Since we have assumed that $\omega(\mathbb{M})<1/(rk)$, it is straightforward to check the existence of $A,B>0$ such that
\begin{equation*}
M_p\le A B^p(p!)^{1/(rk)},\quad p\ge0.
\end{equation*}
This fact, together with standard estimates, guarantee that
\begin{equation}\label{e847}
h_{\mathbb{M}}(\tilde{K}_2|\epsilon|)\le e^{-K'/|\epsilon|^{rk}},
\end{equation}
for some $K'>0$ and all $\epsilon\in\mathcal{E}$.

In view of (\ref{e821}), (\ref{e834}), (\ref{e844}) and (\ref{e847}) we deduce the existence of $\tilde{K}_4,K_5>0$ such that
$$|X_{i+1}(t,z,\epsilon)-X_{i}(t,z,\epsilon)|\le k\exp\left(-\frac{\tilde{K}_4}{|\epsilon|^{rk}}\right)\sum_{\beta\ge0}(K_5|z|)^{\beta},$$
for all $\mathcal{T}\cap D(0,h')$ and $\epsilon\in\mathcal{E}$. From this, one concludes that
$$|X_{i+1}(t,z,\epsilon)-X_{i}(t,z,\epsilon)|\le \tilde{K}_3e^{-\tilde{K}_4/|\epsilon|^{rk}},$$
for some $\tilde{K}_3>0$ and every $t\in\mathcal{T}\cap D(0,h')$, $\epsilon\in\mathcal{E}$ and $z\in D(0,1/(2K_5))$.
\end{proof}

\section{Existence of formal solution in the complex parameter and two level asymptotic expansions}\label{seccionotramas}

\subsection{A general Ramis-Sibuya theorem in two levels}

In the previous section, we have observed (Theorem~\ref{teo445} and Theorem~\ref{teo526}) a different behavior of the difference of two adjacent solutions, depending on the geometry of the problem and the nature of the elements appearing in the equation. This causes the existence of a formal solution of the main problem which can be put as a sum of two formal power series, related to both phenomena.

In the previous work~\cite{lama}, Section 6.1, the first and second authors developed a novel version of Ramis-Sibuya theorem in two Gevrey levels. This result is no longer available in this problem, where more general asymptotics, associated with strongly regular sequences, may appear.

For a reference on the classical version of Ramis-Sibuya theorem, we refer to~\cite{hssi}, Theorem XI-2-3. The next lemma is a general version of Lemma XI-2-6 from \cite{hssi} in the framework of strongly regular sequences.

\begin{lemma}\label{lema586}
Let $\mathbb{M}=(M_{p})_{p\in\N_0}$ be a strongly regular sequence, and let $(\mathcal{E}_{i})_{0\le i\le \nu-1}$ be a good covering in $\C^{\star}$. Assume there exist $f_{1}$,\ldots,$f_{\nu}$ such that:
\begin{itemize}
\item[$(i)$] $f_{\ell}$ is holomorphic in $\mathcal{E}_{\ell-1}\cap \mathcal{E}_{\ell}$ for every $\ell=1,\ldots,\nu$ (where $\mathcal{E}_{\nu}:=\mathcal{E}_{0}$).
\item[$(ii)$] There exist $C_1,C_2>0$ such that
$$|f_{\ell}(\epsilon)|\le C_1 h_{\mathbb{M}}(C_2|\epsilon|),$$
for every $\epsilon\in\mathcal{E}_{\ell-1}\cap \mathcal{E}_{\ell}$, and all $\ell=1,\ldots,\nu$.
\end{itemize}
Then, there exist $\psi_{0}$,\ldots,$\psi_{\nu-1}$ and a formal power series $\hat{\psi}=\sum_{p\ge0}a_{p}\epsilon^p\in\C[[\epsilon]]$ such that
\begin{itemize}
\item[$(i)$] $\psi_{\ell}$ admits $\hat{\psi}$ as its $\mathbb{M}$-asymptotic expansion in $\mathcal{E}_{\ell}$ (see Definition~\ref{defi314}), for all $\ell=0,\ldots,\nu-1$.
\item[$(ii)$] $f_{\ell}(\epsilon)=\psi_{\ell}(\epsilon)-\psi_{\ell-1}(\epsilon)$ for $\epsilon\in \mathcal{E}_{\ell-1}\cap\mathcal{E}_{\ell}$, for every $0\le \ell\le \nu-1$.
\end{itemize}
\end{lemma}
\begin{proof}
We only give details at some strategic points in the proof where it defers from the proof of Lemma XI-2-6 in~\cite{hssi}, which is quite classical and known.

Let $\mathcal{C}_{\ell}:=\{te^{i\theta_{\ell}}:0<t<\tilde{r}_{\mathcal{E}}\}$, for $\ell=0,\ldots,\nu-1$. Here, $\theta_{\ell}$ stands for a fixed argument in $\mathcal{E}_{\ell-1}\cap\mathcal{E}_{\ell}$, and $\tilde{r}_{\mathcal{E}}>0$ is such that $ \mathcal{C}_{\ell}\subseteq (\mathcal{E}_{\ell-1}\cap\mathcal{E}_{\ell})$.

The function $\psi_{\ell}$ defined for every $\ell=0,\ldots,\nu-1$ by
$$\psi_{\ell}(\epsilon):=\frac{-1}{2\pi i}\sum_{h=0}^{\nu-1}\int_{\mathcal{C}_{h}}\frac{f_{h}(\xi)}{\xi-\epsilon}d\xi,$$
can be continued analytically onto $\mathcal{E}_{\ell}$ by path deformation of the integrals involved. These functions satisfy $(ii)$ in the statement. In order to provide the asymptotic behavior of $(i)$, let $0\le\ell\le \nu-1$, and  consider a subsector $T$ of $\mathcal{E}_{\ell}$ which, without loss of generality, we admit to have radius less than $\tilde{r}_{\mathcal{E}}$. One can deform the path $\mathcal{C}_{\ell+1}$ (resp. $\mathcal{C}_{\ell}$) to $\tilde{\mathcal{C}}_{\ell+1}$ (resp. $\tilde{\mathcal{C}}_{\ell}$) without moving the endpoints so that $T$ is contained in the interior of a closed curve $\tilde{\mathcal{C}}_{\ell}+\gamma_{\ell}-\tilde{\mathcal{C}}_{\ell+1}$, where $\gamma_{\ell}$ is a circular arc from $\tilde{r}_{\mathcal{E}}e^{i\theta_{\ell}}$ to $\tilde{r}_{\mathcal{E}}e^{i\theta_{\ell+1}}$. Moreover, one may assume that $\tilde{\mathcal{C}}_{\ell}:=L_{\ell}+\Gamma_{\ell}$, with
$$L_{\ell}:=\{\epsilon=te^{i\omega_{\ell}}:0< t\le r_{\mathcal{E},1}\}$$ and
$$\Gamma_{\ell}:=\{\epsilon=\mu_{\ell}(\tau):0\le \tau<1 \},$$
for some $r_{\mathcal{E},1}<\tilde{r}_{\mathcal{E}}$, some $\theta_{\ell}<\omega_{\ell}$ such that $\omega_{\ell}$ is an argument of $\mathcal{E}_{\ell}$ and where
$$\mu_{\ell}(0)=r_{\mathcal{E},1}e^{i\omega_{\ell}},\quad \mu_{\ell}(1)=\tilde{r}_{\mathcal{E}}e^{i\theta_{\ell}},$$
and  $r_{\mathcal{E},1}\le |\mu_{\ell}(\tau)|<\tilde{r}_{\mathcal{E}}$ for $0\le \tau<1$.

The function $\psi_{\ell}$ can be rewritten in the form
\begin{equation}\label{e615}
\psi_{\ell}(\epsilon)=\frac{-1}{2\pi i}\int_{\tilde{\mathcal{C}}_{\ell}}\frac{f_{\ell}(\xi)}{\xi-\epsilon}d\xi+ \frac{-1}{2\pi i}\int_{\tilde{\mathcal{C}}_{\ell-1}}\frac{f_{\ell+1}(\xi)}{\xi-\epsilon}d\xi+ \frac{-1}{2\pi i}\sum_{h\notin\{\ell-1,\ell\}}\int_{\tilde{\mathcal{C}}_{h}}\frac{f_{h}(\xi)}{\xi-\epsilon}d\xi,
\end{equation}
for $\epsilon\in T$.

There exists a sequence of complex numbers $(a_{p})_{p\ge0}$ such that
\begin{equation}\label{e621a}
\frac{1}{2\pi i}\int_{L_{\ell}}\frac{f_{\ell}(\xi)}{\xi-\epsilon}d\xi=\sum_{p=0}^{N}a_{p}\epsilon^p+\epsilon^{N+1}E_{N+1}(\epsilon),
\end{equation}
for every $\epsilon\in T$ and all $N\in\N_0$, with
$$E_{N+1}(\epsilon):=\frac{1}{2\pi i}\int_{L_{\ell}}\frac{f_{\ell}(\xi)}{\xi^{N+1}(\xi-\epsilon)}.$$
We now provide upper bounds for $|E_{N+1}(\epsilon)|$, $\epsilon\in T$. It is straightforward to check that there exists $0<\theta<\pi/2$ such that
$$|\xi-\epsilon|\ge |\xi|\sin(\theta),\quad \xi\in L_{\ell},\epsilon\in T.$$
This entails
$$
|E_{N+1}(\epsilon)|\le \frac{1}{2\pi}\int_{0}^{r_{\mathcal{E},1}}\frac{|f_{\ell}(te^{i\omega_{\ell}})|}{t^{N+2}\sin(\theta)}dt\le \frac{1}{2\pi\sin(\theta)}\int_{0}^{r_{\mathcal{E},1}}\frac{C_{1}h_{\mathbb{M}}(C_{2}t)}{t^{N+2}}dt.$$
After the change of variable $s=1/t$, one derives
$$|E_{N+1}(\epsilon)|\le\frac{C_1}{2\pi\sin(\theta)}\int_{1/r_{\mathcal{E},1}}^{\infty}h_{\mathbb{M}}(C_2/s)s^{N}ds\le \frac{C_1}{2\pi\sin(\theta)}\int_{0}^{\infty}h_{\mathbb{M}}(C_2/s)s^{N}ds.$$

\noindent \textbf{Remark:} In~\cite{javi1}, Remark 5.8(i), it was proved that given $K_1>0$ there exist $K_2,K_3>0$ such that for every $p\in\N$ one has
$$\int_{0}^{\infty}t^{p-1}h_{\mathbb{M}}(K_{1}/t)dt\le K_2K_3^pM_p.$$

Bearing in mind that $\mathbb{M}$ satisfies $(\mu)$ property (see Definition~\ref{defi310}) one concludes that

$$|E_{N+1}(\epsilon)|\le \frac{C_1}{2\pi\sin(\theta)}K_2K_3^{N+1}M_{N+1}\le \frac{C_1 K_2K_3 AM_1}{2\pi\sin(\theta)}(AK_3)^{N}M_{N},$$
for every $\epsilon\in T$.
One concludes from (\ref{e621a}) that for every subsector $T$ of $\mathcal{E}_{\ell}$ there exist $\Delta_1,\Delta_2>0$ such that
$$\left|\frac{1}{2\pi i}\int_{L_{\ell}}\frac{f_{\ell}(\xi)}{\xi-\epsilon}d\xi-\sum_{p=0}^{N}a_p\epsilon^p\right|\le \Delta_1\Delta_2^{N}M_{N}|\epsilon|^{N+1},\quad \epsilon\in T,N\in\N.$$

Analogous estimates can be obtained for the remaining terms of the sum in (\ref{e615}), and the result is attained.
\end{proof}

One can also generalize in this context the classical notion of multisummability (see \cite{ba}, Chapter 6).

\begin{defin}\label{defi446} Let $(\mathbb{E},||.||_{\mathbb{E}})$ be a complex Banach space, let $\kappa>0$ and let $\mathbb{M} = (M_{p})_{p \geq 0}$ be a strongly regular sequence such that $\omega(\mathbb{M})<1/\kappa$.

Let $\mathcal{E}$ be a bounded open sector centered at 0 with aperture $\pi\omega(\mathbb{M}) + \delta_{2}$ for some
$\delta_{2}>0$ and let $\mathcal{F}$ be a bounded open sector centered at 0 with aperture $\frac{\pi}{\kappa} + \delta_{1}$
for some $\delta_{1}>0$ such that the inclusion $\mathcal{E} \subset \mathcal{F}$ holds.

A formal power series $\hat{f}(\epsilon) = \sum_{n \geq 0} a_{n} \epsilon^{n} \in \mathbb{E}[[\epsilon]]$ is said to be
$(\mathbb{M},\kappa)-$summable on $\mathcal{E}$ if there exist a formal series
$\hat{f}_{2}(\epsilon) \in \mathbb{E}[[\epsilon]]$ which is $\mathbb{M}-$summable on $\mathcal{E}$ with
$\mathbb{M}-$sum $f_{2}: \mathcal{E} \rightarrow \mathbb{E}$ and a second formal series
$\hat{f}_{1}(\epsilon) \in \mathbb{E}[[\epsilon]]$ which is $\kappa-$summable on $\mathcal{F}$ with
$\kappa-$sum $f_{1}: \mathcal{F} \rightarrow \mathbb{E}$ such that
$\hat{f} = \hat{f}_{1} + \hat{f}_{2}$. Furthermore, the holomorphic function
$f(\epsilon) = f_{1}(\epsilon) + f_{2}(\epsilon)$ defined on $\mathcal{E}$ is called the
$(\mathbb{M},\kappa)-$sum of $\hat{f}$ on $\mathcal{E}$.
\end{defin}

\begin{theo}\textbf{(RS)} Let $(\mathbb{E},\left\|\cdot\right\|)$ be a complex Banach space, $(\mathcal{E}_{i})_{0\le i\le \nu-1}$ a good covering in $\C^{\star}$ and fix a strongly regular sequence $\mathbb{M}=(M_{p})_{p\ge0}$. We assume $G_i:\mathcal{E}_{i}\to\mathbb{E}$ is a holomorphic function for all $0\le i\le \nu-1$ and put $\Delta_{i}(\epsilon):=G_{i+1}(\epsilon)-G_i(\epsilon)$ for every $\epsilon\in Z_i:=\mathcal{E}_{i}\cap \mathcal{E}_{i+1}$. Moreover, we assume the next assertions:
\begin{itemize}
\item[$1)$] The functions $G_i(\epsilon)$ are bounded as $\epsilon\in\mathcal{E}_{i}$ tends to 0, for all $0\le i\le \nu-1$.
\item[$2)$] Let $\alpha>0$ and nonempty sets $I_1,I_2\subseteq\{0,\ldots,\nu-1\}$ such that $I_1\cup I_2=\{0,\ldots,\nu-1\}$ and $I_1\cap I_2=\emptyset$.

For every $i\in I_1$ there exist $K_1,M_1>0$ such that
$$\left\|\Delta_i(\epsilon)\right\|_{\mathbb{E}}\le K_1e^{-\frac{M_1}{|\epsilon|^{\alpha}}},\quad \epsilon\in Z_i.$$

In addition to this, for every $i\in I_2$ there exist $K_2,M_2>0$ such that
$$\left\|\Delta_i(\epsilon)\right\|_{\mathbb{E}}\le K_1h_{\mathbb{M}}(M_2|\epsilon|),\quad \epsilon\in Z_i.$$

\end{itemize}

Then, there exists a convergent power series $a(\epsilon)\in\mathbb{E}\{\epsilon\}$ defined on some neighborhood of the origin and $\hat{G}^1(\epsilon),\hat{G}^2(\epsilon)\in\mathbb{E}[[\epsilon]]$ such that $G_i$ can be written in the form
$$G_i(\epsilon)=a(\epsilon)+G_{i}^{1}(\epsilon)+G^{2}_{i}(\epsilon).$$
$G^{1}_{i}(\epsilon)$ is holomorphic on $\mathcal{E}_{i}$ and has $\hat{G}^{1}(\epsilon)$ as its $1/\alpha$-Gevrey asymptotic expansion on $\mathcal{E}_{i}$ for every $i\in I_1$. $G^{2}_{i}(\epsilon)$ is holomorphic on $\mathcal{E}_{i}$ and has $\hat{G}^{2}(\epsilon)$ as its $\mathbb{M}-$asymptotic expansion on $\mathcal{E}_i$, for $i\in I_2$.

Assume moreover that for some integer $i_{0} \in I_{2}$ is such that
$I_{\delta_{1},i_{0},\delta_{2}} = \{ i_{0} - \delta_{1},\ldots,i_{0},\ldots,i_{0}+\delta_{2} \} \subset I_{2}$ for
some integers $\delta_{1},\delta_{2} \geq 0$ and with the property that
\begin{equation}\label{e1022}
 \mathcal{E}_{i_0} \subset S_{\pi/\alpha} \subset \bigcup_{h \in I_{\delta_{1},i_{0},\delta_{2}}} \mathcal{E}_{h}
 \end{equation}
where $S_{\pi/\alpha}$ is a sector with aperture slightly larger than $\pi/\alpha$. Then, the formal
series $\hat{G}(\epsilon)$ is $(\mathbb{M},\alpha)-$summable on $\mathcal{E}_{i_0}$ as stated in Definition~\ref{defi446} and its
$(\mathbb{M},\alpha)-$sum is $G_{i_0}(\epsilon)$.
\end{theo}
\begin{proof}
We define $\Delta^{j}_i(\epsilon)=\Delta_i(\epsilon)\delta_{ij}$ for $j=1,2$, where $\delta_{ij}$ stands for the Kronecker function with value 1 if $i\in I_j$ and 0 otherwise. A direct application of Lemma XI-2-6 in~\cite{hssi} (resp. Lemma~\ref{lema586}) provides that for every $0\le i\le \nu-1$ there exist holomorphic functions $\Psi_{i}^{1}:\mathcal{E}_{i}\to\C$ (resp. $\Psi_{i}^{2}:\mathcal{E}_{i}\to\C$) such that $\Delta_{i}^{j}(\epsilon)=\Psi_{i+1}^{j}(\epsilon)-\Psi_{i}^{j}(\epsilon)$ for every $\epsilon\in Z_i$, $j=1,2$. We put $\Psi_{\nu}^{j}(\epsilon)=\Psi_{0}^{j}(\epsilon)$. In addition to this, there exist formal power series $\sum_{m\ge0}\phi_{m,j}\epsilon^{m}\in\mathbb{E}[[\epsilon]]$, $j=1,2$,
such that for all $0\le \ell\le \nu-1$ and any closed proper subsector $\mathcal{W}\subseteq\mathcal{E}_{\ell}$ with vertex at 0, there exist $\breve{K}_{\ell},\breve{M}_{\ell}>0$ with
$$\left\|\Psi_{\ell}^{1}(\epsilon)-\sum_{m=0}^{N-1}\phi_{m,1}\epsilon^m\right\|_{\mathbb{E}}\le \breve{K}_{\ell}(\breve{M}_{\ell})^{N}N!^{1/\alpha}|\epsilon|^{N},$$
and
$$\left\|\Psi_{\ell}^{2}(\epsilon)-\sum_{m=0}^{N-1}\phi_{m,2}\epsilon^m\right\|_{\mathbb{E}}\le \breve{K}_{\ell}(\breve{M}_{\ell})^{N}M_{N}|\epsilon|^{N},$$
for every $\epsilon\in\mathcal{W}$, and all positive $N$.
The bounded holomorphic function $a_i(\epsilon)=G_{i}(\epsilon)-\Psi_i^1(\epsilon)-\Psi_i^2(\epsilon)$, for every $0\le i\le \nu-1$, and $\epsilon\in\mathcal{E}_i$ is such that $a_{i+1}(\epsilon)=a_i(\epsilon)$ for $\epsilon\in Z_i$. Hence, one can define a holomorphic function $a$ in a neighborhood of the origin which coincides with $a_i$  in $\mathcal{E}_i$ for all $0\le i\le \nu-1$.

The result follows from here for $G_i(\epsilon)=a(\epsilon)+\Psi_i^1(\epsilon)+\Psi_i^2(\epsilon)$, $\epsilon\in\mathcal{E}_i$ and $0\le i\le \nu-1$.

Under the last additional assumption in the statement of the theorem, we can observe that in the decomposition
$G_{i_0}(\epsilon) = a(\epsilon) + G_{i_0}^{1}(\epsilon) + G_{i_0}^{2}(\epsilon)$, the function $G_{i_0}^{1}(\epsilon)$ can
be analytically continued on the sector $S_{\pi/\alpha}$ and has the formal series $\hat{G}^{1}(\epsilon)$ as asymptotic
expansion of Gevrey order $1/\alpha$ on $S_{\pi/\alpha}$ (this is the consequence of the fact that
$\Delta_{h}^{1}(\epsilon)=0$ for $h \in I_{\delta_{1},i_{0},\delta_{2}}$). Hence, $G_{i_0}^{1}(\epsilon)$ is the
$\alpha-$sum of $\hat{G}^{1}(\epsilon)$ on $S_{\pi/\alpha}$. Moreover, we already know that the function
$G_{i_0}^{2}(\epsilon)$ has $\hat{G}^{2}(\epsilon)$ as $\mathbb{M}-$asymptotic expansion on $\mathcal{E}_{i_0}$, meaning that
$G_{i_0}^{2}$ is the $\mathbb{M}-$sum of $\hat{G}^{2}(\epsilon)$ on $\mathcal{E}_{i_0}$. In other words, by Definition~\ref{defi446} above,
the formal series $\hat{G}(\epsilon) = a(\epsilon) + \hat{G}^{1}(\epsilon) + \hat{G}^{2}(\epsilon)$ is
$(\mathbb{M},\alpha)-$summable on $\mathcal{E}_{i_0}$ and its $(\mathbb{M},\alpha)-$sum is
$G_{i_0}(\epsilon)$.

\end{proof}

\textbf{Remark:} In the problem under study, it is sufficient to depart from holomorphic functions $b_{s\kappa_0\kappa_1\beta}^{(i,i+1)}(\epsilon)$ which are holomorphic and bounded in $\mathcal{E}_{i}\cap\mathcal{E}_{i+1}$ and such that admit null $\mathbb{M}-$asymptotic expansion in their domain of definition. Under these initial assumptions, one can apply the novel version of Ramis-Sibuya theorem in the present work in order to handle the problem.

\textbf{Remark:} Observe that the latter statement in the Theorem is feasible under the Assumptions made on the geometry of the problem. If, for example,  $\omega(\mathbb{M})=(4rk)^{-1}$, the aperture of the elements in the good covering can be chosen close but larger than this number. In order not to attain singular directions, so that (\ref{e1022}) holds, one should depart from a subset $\mathcal{S}$ satisfying that $s^2\le4\pi r_1/k$ for every $(s,\kappa_0,\kappa_1)\in\mathcal{S}$.

\section{Existence of formal power series solutions in the complex parameter}

We state the main result in this work, which guarantees the existence of a formal power series in the perturbation parameter which can be split into two formal power series so that the different behavior of the actual solution appears explicitly.

In the remaining results, $\mathbb{E}$ stands for the Banach space of holomorphic functions on the set $(\mathcal{T}\cap D(0,h'))\times D(0, R_0)$ endowed with the supremum norm, where $h'$ and $R_0=\min(\tilde{R}_0,R'_0)$ are given by the constants appearing in Theorems~\ref{teo445} and~\ref{teo526}.

\begin{theo}\label{teo726}
We make assumptions (A) and (B), and suppose that the estimates in (\ref{e325}) and (\ref{e442}) are satisfied. Then, there exists a formal power series
\begin{equation}\label{e691}
\hat{X}(t,z,\epsilon)=\sum_{\beta\ge0}H_{\beta}(t,z)\frac{\epsilon^{\beta}}{\beta!} \in\mathbb{E}[[\epsilon]],
\end{equation}
formal solution of
\begin{equation}\label{e692}
(\epsilon^{r_1}(t^{k+1}\partial_t)^{s_1}+a)\partial_{z}^{S}\hat{X}(t,z,\epsilon)= \sum_{(s,\kappa_0,\kappa_1)\in\mathcal{S}} \hat{b}_{s\kappa_0\kappa_1}(z,\epsilon)t^s (\partial_{t}^{\kappa_0}\partial_{z}^{\kappa_1}\hat{X})(t,z,\epsilon).
\end{equation}
In addition to this, $\hat{X}$ can be written in the form
$$\hat{X}(t,z,\epsilon)=a(t,z,\epsilon)+\hat{X}^{1}(t,z,\epsilon)+ \hat{X}^{2}(t,z,\epsilon),$$
where $a(t,z,\epsilon)\in\mathbb{E}\{\epsilon\}$, $\hat{X}^{1}(t,z,\epsilon), \hat{X}^{2}(t,z,\epsilon)\in\mathbb{E}[[\epsilon]]$. Moreover, for every $0\le i\le \nu-1$ the $\mathbb{E}-$valued function $\epsilon\mapsto X_{i}(t,z,\epsilon)$ constructed in Theorem~\ref{teosol} is given by
$$X_{i}(t,z,\epsilon)=a(t,z,\epsilon)+X^{1}_i(t,z,\epsilon)+X^{2}_i(t,z,\epsilon),$$
where $\epsilon\mapsto X_{i}^{1}(t,z,\epsilon)$ is an $\mathbb{E}-$valued function which admits $\hat{X}^1(t,z,\epsilon)$ as its $s_1/r_1-$Gevrey asymptotic expansion on $\mathcal{E}_{i}$, and where $\epsilon\mapsto X_{i}^{2}(t,z,\epsilon)$ is an $\mathbb{E}-$valued function which admits $\hat{X}^2(t,z,\epsilon)$ as its $\mathbb{M}-$asymptotic expansion on $\mathcal{E}_{i}$.

Under the further assumptions on the good covering $\{ \mathcal{E}_{i} \}_{0 \leq i \leq \nu-1}$ and directions $d_{i}$
introduced in Definition~\ref{defin2} one can guarantee multisummability of the formal solution. Assume there exist $0 \leq i_{0} \leq \nu-1$ and two integers $\delta_{1},\delta_{2} \geq 0$ such that
for all $h \in I_{\delta_{1},i_{0},\delta_{2}} = \{ i_{0} - \delta_{1},\ldots,i_{0},\ldots,i_{0}+\delta_{2} \}$, there are
no singular directions $\frac{\pi(2j+1) + \mathrm{arg}(a)}{ks_{1}}$, for any $0 \leq j \leq ks_{1}-1$, in between
$\gamma_{p}$ and $\gamma_{p+1}$, such that
$$ \mathcal{E}_{i_0} \subset S_{\frac{\pi}{r_{1}/s_{1}}} \subset \bigcup_{h \in I_{\delta_{1},i_{0},\delta_{2}}} \mathcal{E}_{h} $$
where $S_{\frac{\pi}{r_{1}/s_{1}}}$ is a sector with aperture a bit larger than $\frac{\pi}{r_{1}/s_{1}}$.
Then, the formal series $\hat{X}(t,z,\epsilon)$ is $(\mathbb{M},\frac{r_1}{s_1})$-summable on $\mathcal{E}_{i_0}$ and its
$(\mathbb{M},\frac{r_1}{s_1})$-sum is given by $X_{i_0}(t,z,\epsilon)$.

\end{theo}
\begin{proof}
Let $(X_{i}(t,z,\epsilon))_{0\le i\le \nu-1}$ be the finite family of functions constructed in Theorem~\ref{teosol}. For every $0\le i\le \nu-1$, $G_i(\epsilon):=(t,z)\mapsto X_{i}(t,z,\epsilon)$ turns out to be a holomorphic and bounded function from $\mathcal{E}_{i}$ to $\mathbb{E}$. Bearing in mind Theorem~\ref{teo526} (resp. Theorem~\ref{teo445}), the cocycle $\Delta_{i}(\epsilon)=G_{i+1}(\epsilon)-G_i(\epsilon)$ satisfies exponentially flat bounds of Gevrey order $r_1/s_1$ (resp. bounds of the form (\ref{e337})) so that Theorem (RS) guarantees the existence of $\hat{G}(\epsilon),\hat{G}^{1}(\epsilon),\hat{G}^{2}(\epsilon)\in\mathbb{E}[[\epsilon]]$,  $a(\epsilon)\in\mathbb{E}\{\epsilon\}$ and $G_i^1(\epsilon),G_i^2(\epsilon)\in\mathbb{E}(\mathcal{E}_i)$ such that one has the decompositions
\begin{align*}
\hat{G}_i(\epsilon)&=a(\epsilon)+\hat{G}^1(\epsilon)+\hat{G}^2(\epsilon),\\
G_i(\epsilon)&=a(\epsilon)+G_i^1(\epsilon)+G_i^2(\epsilon),
\end{align*}%
where $G^{1}_{i}(\epsilon)$ is a holomorphic function on $\mathcal{E}_{i}$ and admits $\hat{G}^{1}(\epsilon)$ as its Gevrey asymptotic expansion of order $s_1/r_1$ in $\mathcal{E}_{i}$, and where $G^{2}_{i}(\epsilon)$ is a holomorphic function on $\mathcal{E}_{i}$ and admits $\hat{G}^{2}(\epsilon)$ as its $\mathbb{M}-$asymptotic expansion in $\mathcal{E}_{i}$.

Under the additional assumptions in the statement, the Theorem (RS) claims that the formal series $\hat{G}(\epsilon)$ is
$(\mathbb{M},\frac{r_1}{s_1})$-summable on $\mathcal{E}_{i_0}$ and that its
$(\mathbb{M},\frac{r_1}{s_1})$-sum is given by $G_{i_0}(\epsilon)$.

We now show that $\hat{X}$ satisfies (\ref{e692}). For every $0\le i\le \nu-1$, the fact that $G_i^j(\epsilon)$ admits $\hat{G}_i^j(\epsilon)$ as its asymptotic expansion (Gevrey asymptotic expansion for $j=1$ and in the sense of Definition~\ref{defi314} for $j=2$) in $\mathcal{E}_{i}$ implies that
\begin{equation}\label{e1103}
\lim_{\epsilon\to0,\epsilon\in\mathcal{E}_{i}}\sup_{(t,z)\in(\mathcal{T}\cap \{|t|<h'\})\times D(0,R_0)}|\partial_{\epsilon}^{\ell}X_i(t,z,\epsilon)-H_{\ell}(t,z)|=0,
\end{equation}
for all $\ell\ge0$. We derive $\ell<r_1$ times at both sides of equation (\ref{e692}) and let $\epsilon\to0$. From (\ref{e1103}) we get a recursion formula for the coefficients in (\ref{e691}) given by
\begin{align*}
a\partial_{z}^S\left(\frac{H_{\ell}(t,z)}{\ell!}\right)&= \sum_{(s,\kappa_0,\kappa_1)\in\mathcal{S}}\sum_{m=1}^{\ell}\frac{\ell!}{m!(\ell-m)!} \frac{b_{s\kappa_0\kappa_1m}(z)}{m!}\frac{\partial_t^{\kappa_0}\partial_{z}^{\kappa_1}H_{\ell-m}(t,z)}{(\ell-m)!} \\
&-a(t^{k+1}\partial_t)^{s_1}\partial_z^S\left(\frac{H_{\ell-r_1}(t,z)}{(\ell-r_1)!}\right).
\end{align*}
Following the same steps one concludes that the coefficients in $\hat{G}(\epsilon)$ and the coefficients of the analytic solution, written as a power series in $\epsilon$, coincide. This yields $\hat{X}(t,z,\epsilon)$ is a formal solution of (\ref{e394}),(\ref{e395}).


\end{proof}

\end{document}